\documentclass[11pt,twoside]{article}

\usepackage{amsfonts}
\usepackage[mathscr]{euscript}
\usepackage{amsthm}
\newcommand{\lbl}[1]{\label{#1}}

\newtheorem{theorem}{Theorem}[section]
\newtheorem{proposition}{Proposition}[section]
\newtheorem{lemma}{Lemma}[section]
\newtheorem{remark}{Remark}[section]

\newcommand{\be}{\begin{equation}}
\newcommand{\ee}{\end{equation}}
\newcommand\bes{\begin{eqnarray}} \newcommand\ees{\end{eqnarray}}
\newcommand{\bess}{\begin{eqnarray*}}
\newcommand{\eess}{\end{eqnarray*}}
\newcommand\bedd{\bes\left\{\begin{array}{ll}\medskip}
\newcommand\eedd{\end{array}\right.\ees}

\newcommand\lm{\lambda}
\newcommand\ty{\infty}

\newcommand\dd{\displaystyle}

\begin{document}
  \pagestyle{myheadings}
\thispagestyle{empty}

\begin{center}{\Large\bf Free boundary problems for the diffusive competition system }\\[2mm]
{\Large\bf in higher dimension with sign-changing coefficients}\footnote{This work was
supported by NSFC Grant 11371113}\\[4mm]
{\Large Yonggang Zhao}\\[2mm]
Department of Mathematics, Harbin Institute of Technology, Harbin 150001, PR China, and
College of Mathematics and Information Science, Henan Normal University, Xinxiang
453007, PR China \\[3mm]
{\Large Mingxin Wang}\footnote{Corresponding author. E-mail address: mxwang@hit.edu.cn}\\[1mm]
{Natural Science Research Center, Harbin Institute of Technology, Harbin 150080, PR China}
\end{center}

\begin{quote}
\noindent{\bf Abstract.} In this article we investigate two free boundary problems for a Lotka-Volterra competition system in a higher space dimension with sign-changing coefficients. One may be viewed as describing how  two competing species invade if they occupy an initial region, the other describes the dynamical process of a new competitor invading into the habitat of a native species.  For simplicity, it is assumed that the environment is radially symmetric.  The main purpose of this article is to understand the asymptotic behavior of competing species spreading via a free boundary. We derive some sufficient conditions for species spreading success and spreading failure. Moreover, when spreading successfully,  we provide the long time behavior of solutions.


 \end{quote}

 \section{Introduction}
 \setcounter{equation}{0} {\setlength\arraycolsep{3pt}

In this article we study the evolution of positive solutions $(u(t,r),v(t,r),h(t))$, with $r=|x|$ and $x\in\mathbb{R}^N$, to the following  free boundary problems for a Lotka-Volterra type competition system in heterogeneous environment
\begin{equation}\label{1.1}
\left\{\begin{array}{lll}
 u_t-d_1\Delta u=u(a_1(r)-b_1(r)u-c_1(r)v), &t>0, \  0\leq r<h(t),\\[1mm]
 v_t-d_2\Delta v=v(a_2(r)-b_2(r)u-c_2(r)v),&t>0, \  0\leq r<h(t),\\[1mm]
 u_r(t,0)=v_r(t,0)=0,\ \ &t>0, \\[1mm]
 u=v=0,\ h'(t)=-\mu( u_r+\beta v_r), &t>0,\  r=h(t),\\[1mm]
u(0,r)=u_0(r),\ v(0,r)=v_0(r), & 0\leq r\leq h_0=h(0)
 \end{array}\right.
\end{equation}
and
\begin{equation}\label{a1}
\left\{\begin{array}{lll}
u_t-d_1\Delta u=u(a_1(r)-b_1(r)u-c_1(r)v), &t>0, \  0\leq r<h(t),\\[1mm]
v_t-d_2\Delta v=v(a_2(r)-b_2(r)u-c_2(r)v),&t>0, \  0\leq r<\infty,\\[1mm]
u_r(t,0)=v_r(t,0)=0, &t>0, \\[1mm]
h'(t)=-\mu u_r,&t>0,\  r=h(t), \\[1mm]
u(t,r)=0, &t>0,\ h(t)\leq r<\infty,\\[1mm]
u(0,r)=u_0(r), & 0\leq r\leq h_0=h(0),\\[1mm]
v(0,r)=v_0(r),&0\leq r<\infty.
 \end{array}\right.
\end{equation}
In the above two problems, $\Delta u=u_{rr}+\frac{N-1}{r}u_r$; $h_0, \mu, \beta$ and $d_i$ $(i=1,2)$ are given positive constants;  $r=h(t)$ represents the moving boundary to be determined; functions $a_i(r), b_i(r)$, $c_i(r)$  belong to $C^\gamma([0,\infty))\cap L^\infty([0,\infty))$ with $\gamma\in(0,1)$ $(i=1,2)$, and satisfy
\begin{quote}
({\bf H})\ $a_i(r)$ is positive somewhere in $[0,\infty)$, $b_i(r)$ and $c_i(r)$ are positive in $[0,\infty)$, $i=1,2$. Moreover, there exist positive constants $\underline{b}_i$, $\bar{b}_i$, $\underline{c}_i$ and $\bar{c}_i$ such that
$$\underline{b}_i=\inf_{0\leq r<\infty}b_i(r)\leq \sup_{0\leq r<\infty} b_i(r)=\bar{b}_i,\ \ \ \underline{c}_i=\inf_{0\leq r<\infty}c_i(r)\leq \sup_{0\leq r<\infty} c_i(r)=\bar{c}_i.$$
\end{quote}
The initial functions $u_0(r)$ and $v_0(r)$ correspondingly fulfill
\begin{equation}\label{1.2}
\left\{\begin{array}{ll}
u_0(r),v_0(r)\in C^2([0,h_0]),\  u_0(r),  v_0(r)>0 \ {\rm in}\  [0,h_0),\\[1mm]
u_0'(0)=v_0'(0)=u_0(h_0)=v_0(h_0)=0
\end{array}\right.
\end{equation}for problem (\ref{1.1})
and
\begin{equation}\label{a2}
\left\{\begin{array}{ll}
u_0\in C^2([0,h_0]),\ u_0'(0)=u_0(h_0)=0\ {\rm and }\  u_0(r)>0\ {\rm in}\ [0,h_0),\\[1mm]
v_0\in C^2([0,\infty))\cap L^\infty([0,\infty)),\ v_0'(0)=0\ {\rm and }\  v_0(r)>0\ {\rm in}\ [0,\infty)
\end{array}\right.
\end{equation}for problem (\ref{a1}).
We would like to point out that for some conclusions of this article,  one can relax the smoothness hypothesis on the coefficients $a_i(r), b_i(r), c_i(r)$ by only assuming that $a_i(r), b_i(r), c_i(r)\in C([0,\infty))\cap L^\infty([0,\infty))$.

From a biological point of view, problem (\ref{1.1}) may be used to describe how the two new or invasive competing species with population density $(u(t,|x|),v(t,|x|))$ invade if they initially occupy an $N$-dimensional ball $\{|x|<h_0\}$. Both species have a tendency to invade further into their new habitat. The expanding front is represented by the free boundary $\{|x|=h(t)\}$, which is proportional to the normalized population gradient at the spreading front, i.e., $h'(t)=-\mu[u_r(t,h(t))+\beta v_r(t,h(t))]$. Problem (\ref{a1}) describes the dynamical process of a new competitor invading into the habitat of a native species. The new competitor $u(t,|x|)$ initially exists in the ball $\{|x|<h_0\}$, and disperse through random diffusion over an expanding ball $\{|x|<h(t)\}$, whose invading front $\{|x|=h(t)\}$ evolves according to the free boundary condition $h'(t)=-\mu u_r(t,h(t))$. The native species $v(t,|x|)$  undergoes diffusion and growth in the entire available habitat (assumed to be $\mathbb{R}^N$ here). The coefficient functions $a_1(|x|)$ and $a_2(|x|)$ measure the intrinsic growth rates of $u(t,|x|)$ and $v(t,|x|)$, $b_1(|x|)$ and $c_2(|x|)$ represent the intraspecific and $c_1(|x|)$ and $b_2(|x|)$ the interspecific competition rates, and $d_1$ and $d_2$ are the diffusion rates.

Both problem (\ref{1.1}) and (\ref{a1}) are variations of the diffusive Lotka-Volterra type competition model, which has been studied  in detail over a bounded spatial domain or  the entire space $\mathbb{R}^N$. For instance, the dynamical behavior of the problem
\begin{eqnarray*}
\left\{\begin{array}{ll}
u_t-d_1\Delta u=u(a_1-b_1u-c_1v),&t>0,\ x\in \Omega,\\[1mm]
v_t-d_2\Delta v=v(a_2-b_2u-c_2v),&t>0,\ x\in \Omega,\\[1mm]
\frac{\partial u}{\partial n}=\frac{\partial v}{\partial n}=0, &t>0,\ x\in\partial \Omega,\\[1mm]
u(0,x)=u_0(x)>0,\ v(0,x)=v_0(x)>0,&x\in\Omega
\end{array}\right.
\end{eqnarray*}
is widely known  (\cite{CC, KW, Pao}), where $a_i,b_i,c_i$ and $d_i$ (i=1,2) are given positive constants, $\Omega$ is a bounded domain in $\mathbb{R}^N$ $(N\geq1)$ with smooth boundary, $n$ is the outward unit normal vector on $\partial \Omega$. The system can be regarded as  depicting how two competitors evolve in a closed habitat $\Omega$, with no flux across the boundary $\partial \Omega$. Thus their competitive strengths are completely determined by the coefficients in the system. For the entire space problem
\begin{eqnarray*}\label{1.3}
\left\{\begin{array}{ll}
u_t-d_1u_{xx}=u(a_1-b_1u-c_1v),&t>0,\ x\in \mathbb{R},\\[1mm]
v_t-d_2v_{xx}=v(a_2-b_2u-c_2v),&t>0,\ x\in \mathbb{R},
\end{array}\right.
\end{eqnarray*}
there have been many interesting studies on the existence of positive traveling wave solutions (see, e.g., \cite{Kan,OMWM}).

Recently, Guo and Wu (\cite{GW}) investigated the free boundary problem
\begin{equation}\label{aa}
\left\{\begin{array}{lll}
 u_t=u_{xx}+u(1-u-kv), &t>0, \ \ 0<x<s(t),\\[1mm]
 v_t=dv_{xx}+Rv(1-v-hu),&t>0, \ \ 0<x<s(t),\\[1mm]
 u_x=v_x=0,\ \ &t>0, \ \ x=0,\\[1mm]
 u=v=0, \ s'(t)=-\mu(u_x+\beta v_x),\ \ &t>0, \ \ x=s(t), \\[1mm]
 u(0,x)=u_0(x), \ \ v(0,x)=v_0(x),& 0\leq x\leq s_0=s(0)
 \end{array}\right.
\end{equation}
and only focused on the weak competition case: $0<h,k<1$.  It was proved that if $s(\infty)<\frac \pi2\min\{\sqrt{d/R},\, 1\}$, then the two species vanish eventually, i.e.,
  $\lim_{t\to\infty}\|u(t,\cdot),v(t,\cdot)\|_{C([0,s(t)])}=0;$
if $s(\infty)>s^*$, then the two species spread successfully, i.e., $\dd\liminf_{t\to\infty} u(t,\cdot)>0$, $\dd\liminf_{t\to\infty} v(t,\cdot)>0$, where
  $$s^*=\left\{\begin{array}{ll}
  \dd\frac{\pi}{2}\sqrt{\frac d R}\frac 1 {\sqrt{1-h}} \ \ &{\rm if}\ d<R,\\[5mm] \dd\frac\pi 2\frac 1 {\sqrt{1-k}} \ \ &{\rm if}\ d>R,\\[5mm]
   \dd\frac\pi 2\min\left\{\frac 1 {\sqrt{1-k}},\frac 1 {\sqrt{1-h}}\right\} \ \ &{\rm if }\ d=R.
   \end{array}\right.$$
Moreover, they demonstrated that if $d,R,\mu$ and $\beta$ are given and $d\neq R$, then the spreading-vanishing dichotomy can be assured either $h$ or $k$ is small enough.  In addition, the precise asymptotic behavior of $(u,v)$ was provided when the two species spread successfully, i.e.,
$$\lim_{t\to\infty}(u,v)(t,x)=\left(\frac{1-k}{1-hk}, \frac{1-h}{1-hk}\right) $$
uniformly on any compact subset of $[0,\infty)$.
Then, Wang and Zhao (\cite{WZh1}) extended  the results obtained in \cite{GW}, and discussed the long time behavior of the cases with both $0<h<1\leq k$ and $0<k<1\leq h$. Besides, the authors of \cite{WZh1} also disposed of problem (\ref{aa}) with the left boundary condition $u_x=v_x=0$ replaced by $u=v=0$, and derived various interesting results.

In \cite{DL2}, Du and Lin investigated the diffusive competition model (\ref{a1})  in which $a_i(r), b_i(r), c_i(r)$ are given positive constants. In the case that $u$ is an inferior competitor (determined by the reaction terms), they demonstrated that $(u,v)\to (0,v^*)$ as $t\to \infty$, where $(0,v^*)$ is the semitrivial steady-states of the system. When $u$ is a superior competitor, a spreading-vanishing dichotomy were given; moreover, when spreading of $u$ happens, they also presented some rough estimates of the spreading speed.

The primary intention of this article is to generalize the results mentioned above to competition system in heterogeneous environment, where the variable intrinsic growth rate may be ``very negative'' in a `` suitably large region''. We will give some sufficient conditions for species spreading success and spreading failure, and present the long time behavior of solutions when spreading successfully. Though our ideas essentially follow that of \cite{GW,WZh1} and \cite{DL2}, most of the technical proofs here are quite different from and much more involved than the corresponding ones, and some of the results here are proved by completely different methods.

We end the introduction by mentioning some related researches.  In the absence of $v$, the two problems (\ref{1.1}) and (\ref{a1}) both reduce to the diffusive logistic model with a free boundary which has been studied by Du et al. \cite{DG,DG1,DGP}.  Peng and Zhao \cite{PZh} considered a free boundary problem of the diffusive logistic model with seasonal succession. In \cite{WZh2}, Wang and Zhao studied a free boundary problem for a predator-prey model with double free boundaries in one dimension, in which the prey occupies  the whole space but the predator lives in a bounded area at the initial state. Later, Wang \cite{Wjde14, Wcnns14} dealt with the case that both predator and prey live in a bounded area at the initial state. For more mathematical problems with free boundary conditions, we refer the readers to \cite{ChenF,Cui,DLou,ELW,FHX,WZhao,ZW} and some of the references cited therein.

The organization of this article is as follows. In the following three sections, we shall focus on the behavior of problem (\ref{1.1}). The global existence, uniqueness and estimate of solutions $(u,v,h)$ for problem (\ref{1.1}) are established in Section \ref{sec2}. In Section \ref{sec3}, we give some sufficient conditions of spreading and vanishing. Section \ref{sec4} is devoted to the long time behavior of $(u,v)$ for the spreading case. In Section \ref{sec5}, we explain how the techniques for (\ref{1.1}) can be modified to discuss  problem (\ref{a1}), where the free boundary is only determined by the new competitor $u$.

\section{Global existence, uniqueness and estimate of solutions to (\ref{1.1})}\label{sec2}
\setcounter{equation}{0}
In this section we present the global existence, uniqueness and estimate of the solution $(u,v,h)$ to problem (\ref{1.1}).
\begin{theorem}\label{t2.1}
For any given $u_0$ and $v_0$ satisfying {\rm(\ref{1.2})}, problem {\rm(\ref{1.1})} has a unique global solution $(u,v,h)$, and
\begin{equation}\label{2.1}
(u,v,h)\in C^{1+\frac{\gamma}{2},2+\gamma}( Q)\times C^{1+\frac{\gamma}{2},2+\gamma}( Q)\times C^{1+\frac{1+\gamma}{ 2}}((0,\infty)),
\end{equation}
where $Q=\{(t,r)\in \mathbb{R}^2:t\in(0,\infty),\ r\in[0,h(t)]\}$. Moreover, there exist  positive constants $K$ and $K_1$ dependent on $d_i,\mu, \beta, \underline{b}_i, \bar b_i, \underline{c}_i, \bar{c}_i, h_0$ and $\|a_i,u_0,v_0\|_\infty$ such that
\begin{equation}\label{2.2}
0<u(t,r)\leq K, \ 0<v(t,r)\leq K,\ 0<h'(t)\leq K,\ \forall\  t>0,\ 0\leq r<h(t);
\end{equation}
\begin{equation}\label{2.3}
\|u(t,\cdot),v(t,\cdot)\|_{C^1([0,h(t)])}\leq K_1,\ \forall \ t\geq1;\ \ \|h'\|_{C^{\frac{\gamma}{2}}([n+1,n+3])}\leq K_1, \ \forall\ n\geq0.
\end{equation}
\end{theorem}
\begin{proof}
The proof can be done by adapting analogous methods in \cite{DL2, GW, Wang14, WZhao}. For the reader's convenience,  we provide the arguments with obvious modifications, which is divided into the following several steps.

{\it Step 1.}
Notice that $a_i(r)$, $b_i(r)$ and $c_i(r)$ belong to $L^\infty([0, \infty))$. Similar to  the argument of Lemma 2.1 in \cite{GW} we can apply the contraction mapping theorem to show that problem (\ref{1.1}) has a unique local solution $(u,v,h)\in C^{\frac{1+\nu}{2},1+\nu}( Q_\tau)\times C^{\frac{1+\nu}{2},1+\nu}(Q_\tau)\times C^{1+\frac{\nu}{2}}(0,\tau)$ for any $\nu\in(0,1)$ and some $\tau>0$ small enough, where $Q_\tau=\{(t,r):t\in(0,\tau],\ r\in[0,h(t)]\}$.

Then we promote the regularity of the solution $(u,v,h)$. To do this, define the transformations
\begin{equation}\label{2.4}
y=x/h(t),\ \ s=|y|, \ \ w(t,s)=u(t,r),\ \ z(t,s)=v(t,r).
\end{equation}
By elementary calculus one can verify that
\begin{equation}\label{2.5}
\left\{\begin{array}{lll}
 w_t-d_1\zeta(t)\Delta_sw-\xi(t,s)w_s=w(\tilde a_1-\tilde b_1w-\tilde c_1z), &0<t<\tau, \ \ 0\leq s<1,\\[1mm]
 z_t-d_2\zeta(t)\Delta_sz-\xi(t,s)z_s=v(\tilde a_2-\tilde b_2w-\tilde c_2z),&0<t<\tau, \ \ 0\leq s<1,\\[1mm]
 w_s(t,0)=0,\ w(t,1)=0,\ \ &0\le t<\tau, \\[1mm]
 z_s(t,0)=0,\ z(t,1)=0,\ \ &0\le t<\tau,\\[1mm]
 w(0,x)=u_0(h_0s), \ \ z(0,x)=v_0(h_0s),& 0\leq s\leq 1,
 \end{array}\right.
\end{equation}
where $\zeta(t)=h^{-2}(t)$, $\xi(t,s)=h'(t)s/h(t)$, $\tilde g=\tilde g(t,s)= g(h(t)s)$, $g$ may be one of the functions $a_i, b_i, c_i$ (i=1,2). Obviously, (\ref{2.5}) is an initial and boundary problem with fixed boundary condition. Note that $a_i(r), b_i(r)$ and $c_i(r)$ $(i=1,2)$ belong to $C^\gamma([0,\infty))$,
$(u,v,h)\in C^{\frac{1+\nu}{2},1+\nu}( Q_\tau)\times C^{\frac{1+\nu}{2},1+\nu}(Q_\tau)\times C^{1+\frac{\nu}{2}}(0,\tau)$. For any $0<\varepsilon\ll1$, applying Theorem 10.1 of \cite{LSU} to problem (\ref{2.5}) on $[\varepsilon, \tau]\times{[\varepsilon,1]}$ and $[\varepsilon,\tau]\times[0,1-\varepsilon]$, respectively, we can derive
$$w,\,z\in C^{1+\frac{\gamma}2,\,2+\gamma}\big([\varepsilon,\tau]\times[\varepsilon,1]\big)
  \cap C^{1+\frac{\gamma}2,\,2+\gamma}\big([\varepsilon,\tau]\times[0,1-\varepsilon]\big).$$
This combined with (\ref{2.4}) results in
$$u,\,v\in C^{1+\frac{\gamma}2,\,2+\gamma}\big([\varepsilon,\tau]\times[\varepsilon h(t),h(t)]\big)
  \cap C^{1+\frac{\gamma}2,\,2+\gamma}\big([\varepsilon,\tau]\times[0,(1-\varepsilon)h(t)]\big).$$
Owing to the arbitrariness of $\varepsilon$ we have
$u,\,v\in C^{1+\frac{\gamma}2,\,2+\gamma}((0,\tau]\times[0,h(t)])$,
which implies $\ u_r,\,v_r\in C^{\frac{1+\gamma}2,\,1+\gamma}((0,\tau]\times[0,h(t)]).$
Therefore, by means of the free boundary condition $h'(t)=-\mu(u_r+\beta v_r)$, it can be deduced that $h'\in C^{\frac{1+\gamma}2}((0,\tau])$.

{\it Step 2.} We assert that if $(u,v,h)$ is a solution of problem (\ref{1.1}) defined in $[0,\tau]$ for some $\tau\in(0,\infty)$, then there exists a positive constant $K_2$ independent of $\tau$ such that these results of (\ref{2.2}) hold for $t\in(0,\tau)$ and $r\in(0,h(t))$ instead of $t>0$ and $r\in(0,h(t))$.

As $u,v>0$ for $0<r<h(t)$ and $u=v=0$ for $r=h(t)$, we have $u_r(t,h(t)), v_r(t,h(t))\le 0$ and so $h'(t)\ge 0$. In view of
\cite[Lemma 2.6]{Lie}, it can be deduced that $u_r(t,h(t)), v_r(t,h(t))<0$. Therefore,  $h'(t)=-\mu( u_r+\beta v_r)>0$ for $0<t\le\tau$.

Let $\bar u(t)$ be the solution of $u'=u(\|a_1(r)\|_{L^\infty}-\underline{b}_1 u)$ with $u(0)=\|u_0\|_{L^\infty}$. The comparison principle allows us to conclude that $u(t,r)\leq\bar u(t)\leq\max\{\|a_1\|_{\infty}/\underline{b}_1,\|u_0\|_{\infty}\}$ for all $t\in(0,\tau]$ and $r\in[0,h(t)]$. Analogously, we can obtain that $v(t,r)\leq\max\{\|a_2\|_{\infty}/\underline{c}_2,\|v_0\|_{\infty}\}$ for all  $t\in(0,\tau]$ and $r\in[0,h(t)]$.

The proof of remaining results for our assertion is similar to that of corresponding part of Lemma 2.2 in \cite{GW} (or Theorem 2.4 in \cite{DL2}).

{\it Step 3.} By invoking the conclusions of {\it Step 1} and {\it Step 2}, we can apply the contradiction argument to achieve the existence of global solution. When the solution to problem (\ref{1.1}) exists globally,  the procedure of the previous two steps can be still done for arbitrary $\tau\in(0,\infty)$. This finishes the proof of (\ref{2.1}) and (\ref{2.2}).

{\it Step 4.} We show (\ref{2.3}). For any integer $n\geq0$, define $w^n(t,s)=w(t+n,s)$,  $z^n(t, s)=z(t+n,s)$ and $h^n(t)=h(t+n)$. It is easy to check that $(w^n,z^n,h^n)$ satisfies
\begin{eqnarray*}
\left\{\begin{array}{lll}
 w^n_t-d_1\zeta^n(t)\Delta_s w^n-\xi^n(t,s)w^n_s=f^n, &0<t\leq3, \ \ 0\leq s<1,\\[1mm]
 z^n_t-d_2\zeta^n(t)\Delta_s z^n-\xi^n(t,s)z^n_s=g^n,&0<t\leq3, \ \ 0\leq s<1,\\[1mm]
 w^n_s(t,0)=0,\ w^n(t,1)=0,\ \ &0\leq t\leq3, \\[1mm]
 z^n_s(t,0)=0,\ z^n(t,1)=0,\ \ &0\leq t\leq3,\\[1mm]
 w^n(0,x)=u(n,h^n(0)s), \ \ z^n(0,x)=v(n,h^n(0)s),& 0\leq s\leq 1,
 \end{array}\right.
\end{eqnarray*}
where $\zeta^n(t)=\zeta(t+n)$, $\xi^n(t,s)=\xi(t+n,s)$, $f^n=w^n(a_1(h^n(t)s)-b_1(h^n(t)s)w^n-c_1(h^n(t)s)z^n)$ and $g^n=z^n(a_2(h^n(t)s)-b_2(h^n(t)s)w^n-c_2(h^n(t)s)z^n)$.

Combining estimate (\ref{2.2}) with the assumptions on $h_0$, $a_i(r)$, $b_i(r)$ and $c_i(r)$, we can deduce that $\zeta^n$, $\xi^n$, $f^n$ and $g^n$ are uniformly bounded on $n$.  Besides, $w^n(t,1)=z^n(t,1)=0$. Therefore, employing the interior $L^p$ estimate  (see \cite[Theorems 7.15 and 7.20]{Lie}) and embedding theorem yields that there is a positive constant $K_1$ independent of $n$ such that $\|w^n,z^n\|_{C^{\frac{1+\gamma}2,1+\gamma}([1,3]\times[0,1])}\leq K_1$ for any $n\geq0$. This implies that for any $n\geq0$, $\|w,z\|_{C^{\frac{1+\gamma}2,1+\gamma}(E_n)}\leq K_1$ with $E_n=[n+1,n+3]\times[0,1]$. In view of estimate (\ref{2.2}), the transformation (\ref{2.4}) and the free boundary condition, it is not difficult to get that $\|h'\|_{C^{\frac{\gamma}{2}}([n+1,n+3])}\leq K_1$ for all $n\geq0$, i.e., the second estimate of (\ref{2.3}). Because these rectangles $E_n$ overlap and $K_1$ is independent of $n$, we have $\|w,z\|_{C^{0,1}([1,\infty)\times[0,1])}\leq K_1$. And then the first estimate of (\ref{2.3}) is obtained immediately.

This completes the proof of Theorem \ref{t2.1}.
\end{proof}
It is readily seen from the last estimate of (\ref{2.2}) that the free boundary $h(t)$ is strictly monotone increasing. Thus there exists $h_\infty\in(0,\infty]$ such that  $\lim_{t\to\infty}h(t)=h_\infty$.

\section{Spreading-vanishing criteria for problem (\ref{1.1})}\label{sec3}
\setcounter{equation}{0}
To discuss the asymptotic behavior of $u$ and $v$ for vanishing case ($h_\infty<\infty$), we first give the following proposition.
\begin{proposition}\label{pr3.1}
Let $d$, $\nu$, $\sigma$ and $g_0$ be positive constants and $C$ be any real number. Suppose $$w_0\in C^2([0,g_0]),\ w_0'=w_0(g_0)=0,\ w_0(r)>0 {\rm\; in\;}[0,g_0).$$ Assume that $g(t)\in C^{1+\frac \sigma 2}([0,\infty)),\ g(t)>0\ {for}\ 0\leq t<\infty,\ \lim_{t\to\infty} g(t)=g_\infty<\infty, \ \lim_{t\to\infty} g'(t)=0;$ and that $ w\in C^{\frac{1+\sigma}2,1+\sigma}([0,\infty)\times[0,g(t)])$, $ w(t,r)>0$ for $0\leq t<\infty$ and $0\leq r<g(t)$,
 $\|w(t,\cdot)\|_{C^1[0,\,g(t)]}\leq M$ for any $t\geq1$ and some $M>0$.
If $(w,g)$ satisfies
\begin{eqnarray*}
\left\{\begin{array}{lll}
 w_t-d\Delta w\geq Cw, &t>0,\ 0\leq r<g(t),\\[1mm]
 w_r=0,\ \ \ &t>0, \ r=0,\\[1mm]
 w=0,\ g'(t)\geq-\nu w_r, \ &t>0,\ r=g(t),\\[1mm]
 w(0,r)=w_0(r), \ &0\leq r\leq g_0=g(0),
 \end{array}\right.
\end{eqnarray*}
then
  $\lim_{t\to\infty}\,\max_{0\leq r\leq g(t)}w(t,r)=0.$
\end{proposition}
\begin{proof}
The proof of this result is essentially the same as {\it Step 1} of Theorem 3.1 in \cite{ZW},  so we leave out the details.
\end{proof}
\begin{theorem}\label{the3.1}
Assume that $(u,v,h)$ is the solution of problem {\rm (\ref{1.1})}. If $h_\infty<\infty$, then
$$
\lim_{t\to\infty}\|u(t,\cdot),v(t,\cdot)\|_{C([0,h(t)])}=0.
$$
This result indicates that when the two species do not spread successfully, then they must vanish eventually.
\end{theorem}
\begin{proof}
By the last estimate of (\ref{2.3}) we know that $\|h'\|_{C^{\frac{\gamma}{2}}([1,\infty))}\leq K_1$.  Combining this with $h'(t)>0$ and $h_\infty<\infty$ implies $h'(t)\to 0$ as $t\to\infty$.

Due to Hypothesis ({\bf H}) and estimate (\ref{2.2}), there exists positive constant $M_1$ such that  $|a_1-b_1u-c_1v|\leq M_1$ and $|a_2-b_2u-c_2v|\leq M_1$. From the proof of Theorem \ref{t2.1} we can find that $u_r(t,h(t))<0$ and $v_r(t,h(t))<0$. Thus it follows that
\begin{eqnarray*}
\left\{\begin{array}{lll}
 u_t-d_1\Delta u\geq -M_1u, &t>0,\  0\leq r<h(t),\\[1mm]
 u_r=0,\ \ \ &t>0, \  r=0,\\[1mm]
 u=0,\  h'(t)\geq-\mu u_r, \ &t>0,\ r=h(t),\\[1mm]
 u(0,r)=u_0(r), \ &0\leq r\leq h_0.
 \end{array}\right.
\end{eqnarray*}
By virtue of (\ref{2.2}), (\ref{2.3}) and Proposition \ref{pr3.1} it is derived that  $\lim_{t\to\infty}\,\max_{0\leq r\leq h(t)}u(t,r)=0.$
In the same way we immediately get $\lim_{t\to\infty}\,\max_{0\leq r\leq h(t)}v(t,r)=0.$

This proof is completed.
\end{proof}

For any given $\ell$, let $\lambda_1(\ell;q,d)$ denote the principle eigenvalue of the problem
\begin{equation}\label{3.1}
\left\{\begin{array}{ll}
-d\Delta\psi -q(|x|)\psi=\lambda\psi,&x\in B_\ell,\\[1mm]
\psi=0,&x\in \partial B_\ell,
\end{array}
\right.
\end{equation}
where $q\in L^\infty([0,\infty))$ and $d$ is a positive constant, $B_\ell$ represents the ball of $\mathbb{R}^N$ centered at the origin with radius $\ell$. In the sequel we will use $\lambda_1(h_0;a_1,d_1)$, $\lambda_1(h_0;a_2,d_2)$, $\lambda_1(h_\infty;a_1,d_1)$, $\lambda_1(h_\infty;a_2,d_2)$  to denote the first eigenvalue of problem (\ref{3.1}) with $\ell$, $q$, $d$ replaced by the corresponding data, respectively. Notice that $q$ is bounded and the boundary condition is $\psi=0$. Then the following Proposition \ref{pr3.2} is well known (see, e.g., \cite{CC, Ni, W2}).

\begin{proposition}\label{pr3.2}
{\rm(i)} $\lambda_1(\ell;q,d)$ depends continuously on  $\ell$, $q$ and $d$;

{\rm(ii)} $\lambda_1(\ell;q,d)$ is strictly decreasing in $\ell$ and $q(|x|)$, strictly increasing in $d$;

{\rm (iii)} $\lim_{\ell\to0^+}\lambda_1(\ell;q,d)=\lim_{d\to\infty}\lambda_1(\ell;q,d)=\infty$,
$\lim_{d\to0^+}\lambda_1(\ell;q,d)=-\max_{[0,\ell]}q(|x|)$.
\end{proposition}

\begin{proposition}\lbl{p3.2}\, Assume that the function $q(x)$ satisfies one of the following assumptions:
  \vspace{-1mm}\begin{quote}
 {\bf(A1)}\, There exist a constant $\beta>0$ and two sequences $\{R_n\}$, $\{r_n\}$ satisfying $R_n>r_n>0$ and $R_n-r_n\to\ty$ as $n\to\ty$, such that $q(|x|)\ge\beta$ for $r_n\le|x|\le R_n$;\\
  {\bf(A2)}\, There exist three constants $\beta>0$, $k>1$, $-2<\gamma\le 0$ and a sequence $\{r_n\}$ satisfying $r_n\to\ty$ as $n\to\ty$, such that $q(|x|)\ge\beta |x|^\gamma$ for $r_n\le|x|\le kr_n$.
 \vspace{-1mm}\end{quote}
Then for any given $d>0$, there exists a unique $\ell_0=\ell_0(d)>0$ such that $\lm_1(\ell_0;q,d)=0$. Hence, $\lm_1(\ell;q,d)<0$ for all $\ell>\ell_0$.
\end{proposition}

This proposition can be proved by the similar method to that of \cite[Remark 3.1]{Wang14}. The details are omitted here.

\begin{lemma}\label{lem3.1}
Let $\Lambda=\min\{\lambda_1(h_\infty;a_1,d_1),\lambda_1(h_\infty;a_2,d_2)\}$. If $h_\infty<\infty$, then $\Lambda\geq0$.
\end{lemma}
\begin{proof} Here we use the contradiction argument. Assume that the result is false, then we have $\lambda_1(h_\infty;a_1,d_1)<0$ or $\lambda_1(h_\infty;a_2,d_2)<0$.

If $\lambda_1(h_\infty;a_1,d_1)<0$, by the continuity of $\lambda_1(h_\infty;a_1,d_1)$ with respect to $a_1(r)$, one can take sufficiently small $\varepsilon>0$  such that $\lambda_1(h_\infty;a_1-c_1\varepsilon,d_1)<0$. It follows from Theorem \ref{the3.1} that $\lim_{t\to\infty}\|v(t,\cdot)\|_{C([0,h(t)])}=0$ and $\lim_{t\to\infty}\|u(t,\cdot)\|_{C([0,h(t)])}=0.$   For any given  $\varepsilon$, in light of Proposition \ref{pr3.2},  there is $T\gg1$ such that $\lambda_1(h(T);a_1-c_1\varepsilon,d_1)<0$ and $v(t,r)\leq\varepsilon$ for all $t\geq T$ and $0\leq r\leq h(t)$. Let $w(t,r)$ be the unique positive solution of  the initial boundary problem
\begin{eqnarray*}
\left\{\begin{array}{ll}
w_t-d_1\Delta w=w(a_1(r)-c_1(r)\varepsilon-b_1(r)w),&t>T, \ 0\leq r<h(T),\\[1mm]
w_r(t,0)=0=w(t,h(T)),&t\geq T,\\[1mm]
w(T,r)=u(T,r),&0\leq r\leq h(T).
\end{array}
\right.
\end{eqnarray*}
By use of the comparison principle, it can be deduced that $u(t,r)\geq w(t,r)$ for all $t\geq T$ and $0\leq r\leq h(T)$. As $\lambda_1(h(T);a_1-c_1\varepsilon,d_1)<0$, we know that $w(t,r)\to W(r)$ as $t\to \infty$ uniformly on $[0,h(T)]$, where $W(r)$ is the unique positive solution of
\begin{eqnarray*}
\left\{\begin{array}{ll}
-d_1\Delta W=W(a_1(r)-c_1(r)\varepsilon-b_1(r)W),&0\leq r<h(T),\\[1mm]
W_r(0)=0=W(h(T)).
\end{array}
\right.
\end{eqnarray*}
As a result, $\liminf_{t\to\infty}u(t,r)\geq W(r)>0$ in $(0,h(T))$, which brings about a contradiction with the fact $\lim_{t\to\infty}\|u(t,\cdot)\|_{C([0,h(t)])}=0.$

If $\lambda_1(h_\infty;a_2,d_2)<0$, similar to the above, there exist $0<\delta\ll1$ and $\tau\gg1$ such that $\lambda_1(h(\tau);a_2-b_2\delta,d_2)<0$ and $u(t,r)\leq\delta$ for all $t\geq \tau$ and $0\leq r\leq h(t)$. Let $z(t,r)$ and $Z(r)$ denote the unique positive solution of
\begin{eqnarray*}
\left\{\begin{array}{ll}
z_t-d_2\Delta z=z(a_2(r)-b_2(r)\delta-c_2(r)z),&t>\tau, \ 0\leq r<h(\tau),\\[1mm]
z_r(t,0)=0=z(t,h(\tau)),&t\geq \tau,\\[1mm]
w(\tau,r)=v(\tau,r),&0\leq r\leq h(\tau)
\end{array}
\right.
\end{eqnarray*}
and the corresponding stationary problem
\begin{eqnarray*}
\left\{\begin{array}{ll}
-d_2\Delta Z=Z(a_2(r)-b_2(r)\delta-c_2(r)Z),&0\leq r<h(\tau),\\[1mm]
Z_r(0)=0=Z(h(\tau)),
\end{array}
\right.
\end{eqnarray*}respectively. In the same way as above, we have $\liminf_{t\to\infty}v(t,r)\geq\liminf_{t\to\infty}z(t,r)=Z(r)>0$ in $(0,h(\tau))$, and then derive a paradox. This finishes the proof.
\end{proof}

Next, we present a comparison principle which will help us to establish  the two species  vanishing under certain conditions. Its proof can be accomplished in a similar way to Lemma 5.1 in \cite{GW} (also see Lemma 3.5 of \cite{DL}), hence we omit the details.
\begin{proposition} \label{pr3.3}
Let $(u,v,h)$ be a solution of problem {\rm(\ref{1.1})}. Assume that $(\bar u,\bar v,\bar h)\in [C(\overline \mathcal{D})\cap C^{1,2}(\mathcal{D})]^2\times C^1([0,\infty))$ with $\mathcal{D}=\{(t,r):\,t>0,0\leq r<\bar h(t)\}$, satisfying
\begin{eqnarray*}
\left\{\begin{array}{lll}
 \bar u_t-d_1\Delta \bar u\geq \bar u(a_1(r)-b_1(r)\bar u), &t>0, \  0\leq r<\bar h(t),\\[1mm]
 \bar v_t-d_2\Delta \bar v\geq \bar  v(a_2(r)-c_2(r)\bar v),&t>0, \  0\leq r<\bar h(t),\\[1mm]
\bar u_r=0,\ \bar v_r=0,\ \ &t>0, \ r=0,\\[1mm]
 \bar u=\bar v=0, \ \bar h'(t)\geq-\mu(\bar u_r+\beta \bar v_r),\ \ &t>0,  \ r=\bar h(t), \\[1mm]
 \bar u(0,r)\geq0, \ \bar v(0,r)\geq0,& 0\leq r\leq \bar h(0).
 \end{array}\right.
\end{eqnarray*}
If $h_0\leq \bar h(0)$, $u_0(r)\leq \bar u(0,r)$, $v_0(r)\leq\bar v(0,r)$ for all $r\in[0,h_0]$, then
$$h(t)\leq\bar h(t)\ {\rm for\ all}\ t\geq0,\ u(t,r)\leq\bar u(t,r),\ v(t,r)\leq\bar v(t,r)\ {\rm for\ all}\ t\geq0,\ 0\leq r\leq h(t).$$
\end{proposition}
\begin{lemma}\label{lem3.2}
Assume that $\lambda_1(h_0;a_1,d_1)>0$ and $\lambda_1(h_0;a_2,d_2)>0$. Then there exists a positive number $\mu_0$ dependent on $h_0,d_1,d_2,\beta$, $a_1,a_2,u_0$ and $v_0$, such that $h_\infty<\infty$ for any $0<\mu<\mu_0$.
\end{lemma}
\begin{proof}
Inspired by \cite{DG,ZX}, we are going to construct the suitable upper solutions and apply Proposition \ref{pr3.3} to derive the desired result.

Let $\psi_1(r)$, $\psi_2(r)$ be the positive eigenfunctions corresponding to $\lambda_1^1:=\lambda_1(h_0;a_1,d_1)$, $\lambda_1^2:=\lambda_1(h_0;a_2,d_2)$, respectively, and $\|\psi_1(r)\|_\infty=\|\psi_2(r)\|_\infty=1$. Define
\begin{eqnarray*}
&\bar h(t)=h_0(1+2\delta-\delta e^{-\sigma t}),\ t\geq0;\ \ \  s=s(t,r)=\frac{h_0r}{\bar h(t)},\ \ t\geq0,\ 0\leq r\leq \bar h(t);&\\
&\bar u(t,r)=Me^{-\sigma t}\psi_1(s),\ \ \bar v(t,r)=Me^{-\sigma t}\psi_2(s),\ \ t\geq0,\ 0\leq s\leq h_0,&
\end{eqnarray*}
where $\delta,\sigma,M$ are positive constants to be determined later.

We first point out that $\psi_i'(r)<0$, $i=1,2$ on $[h_0-\eta,h_0]$ for some $\eta>0$, and $\psi_i(r)>0$ on $[0,h_0-\eta]$. Thus it is not difficult to manifest that there exists a positive number $M_1$ such that  for $0\leq r<h_0$,
\begin{equation}\label{3.2}
r\psi_i'(r)\leq M_1\psi_i(r),\ i=1,2.
\end{equation}
On the other hand, due to the continuity of $a_1(r)$ and $a_2(r)$ on $[0,3h_0]$, we can verify that for any given $0<\varepsilon<1,$ there exists  $0<\delta_0\ll1$ such that for any $0<\delta\leq\delta_0$,
\begin{equation}\label{3.3}
\left|\frac{a_1(s)h_0^2}{\bar h^2(t)}-a_1(r)\right|\leq\varepsilon,\ \ \ \left|\frac{a_2(s)h_0^2}{\bar h^2(t)}-a_2(r)\right|\leq{\varepsilon},\ \ \forall\ t>0,\ 0\leq r<\bar h(t).
\end{equation}

By virtue of (\ref{3.2}), (\ref{3.3}) and the assumption $\lambda_1^1>0$, elementary computations generate
\begin{eqnarray}\label{3.4}
&&\bar u_t-d_1\Delta \bar u-\bar u(a_1(r)-b_1(r)\bar u) \nonumber\\[1mm]
&=&\bar u\left[-\sigma-\frac{s\psi_1'(s)\bar h'(t)}{\psi_1(s)\bar h(t)}-\frac{d_1\psi_1''(s)h_0^2}{\psi_1(s)\bar h^2(t)}-\frac{d_1(N-1)\psi_1'(s)h_0}{r\psi_1(s)\bar h(t)}-a_1(r)\right]+b_1(r)\bar u^2 \nonumber\\[1mm]
  &\geq& \bar u\left(-\sigma-\frac{s\psi_1'(s)h_0\delta\sigma e^{-\sigma t}}{\psi_1(s)\bar h(t)}-a_1(r)+\frac{h_0^2a_1(s)}{\bar h^2(t)}+\frac{h_0^2\lambda_1^1}{\bar h^2(t)}\right) \nonumber\\[1mm]
   &\geq& \bar u\left(-\sigma-M_1h_0\sigma-\varepsilon+{\lambda_1^1}/{9}\right)>0,\ \ \ \ \;\;\forall\ t>0,\ 0\leq r<\bar h(t)
\end{eqnarray}
provided $0<\varepsilon,\sigma\ll1$. Similarly, it  can be deduced that for all $t>0$ and $0\leq r<\bar h(t)$,
\begin{eqnarray}\label{3.5}
\bar v_t-d_2\Delta\bar v-\bar v(a_2(r)-c_2(r)\bar v)
  &\geq& \bar v\left(-\sigma-\frac{s\psi_2'(s)h_0\delta\sigma e^{-\sigma t}}{\psi_2(s)\bar h(t)}-a_2(r)+\frac{h_0^2a_2(s)}{\bar h^2(t)}+\frac{h_0^2\lambda_1^2}{\bar h^2(t)}\right) \nonumber\\[1mm]
   &\geq& \bar v\left(-\sigma-M_1h_0\sigma-\varepsilon+{\lambda_1^2}/{9}\right)>0
\end{eqnarray}
provided that $0<\varepsilon,\sigma\ll1$. For fixed $0<\delta\leq\delta_0$ and $0<\varepsilon,\sigma\ll1$, we can select sufficiently large positive number $M$ such that
\begin{equation}\label{c7}
u_0(r)\leq M \psi_1(r)=\bar u(0,r),\ \ v_0(r)\leq M \psi_2(r)=\bar v(0,r),\ \forall \ 0\leq r\leq h_0.
\end{equation}

Apparently,
\begin{equation}\label{c8}
\bar u_r(t,0)=0,\ \ \bar v_r(t,0)=0,\ \forall\ t>0.
\end{equation}
On the other hand,  it is easy to show that $\bar h'(t)=h_0\delta\sigma e^{-\sigma t}$ and
\begin{eqnarray*}
-\mu(\bar u_r+\beta\bar v_r)(t,\bar h(t))&=& -\mu\left[Me^{-\sigma t}\frac{h_0}{\bar h(t)}(\psi_1'(h_0)+\beta\psi_2'(h_0))\right] \\[1mm]
&\leq&\mu Me^{-\sigma t}(1+\beta)\max\{|\psi_1'(h_0)|,|\psi_2'(h_0)|\}
\end{eqnarray*}
since $\psi_1'(h_0),\psi_2'(h_0)<0$. Therefore, there is a positive real number $\mu_0$ such that for all $ 0<\mu\leq\mu_0,$
\begin{equation}\label{c9}
\bar h'(t)\geq -\mu(\bar u_r+\beta\bar v_r)(t,\bar h(t)),\ \ \forall \ t>0.
\end{equation}
Additionally, it is obvious that
\begin{equation}\label{3.9}
\bar u(t,\bar h(t))=\bar v(t,\bar h(t))=0,\ \ \forall\ t>0.
\end{equation}

Taking into account (\ref{3.4})-(\ref{3.9}), by means of  Proposition \ref{pr3.3} we can derive
$$h(t)\leq\bar h(t),\ u(t,r)\leq\bar u(t,r),\ v(t,r)\leq\bar v(t,r),\ \forall\ t\geq0,\ 0\leq r\leq h(t).$$
As a consequence, $h_\infty\leq \bar h(\infty)=h_0(1+2\delta)$ for any $0<\mu\leq\mu_0.$
This proof is completed.
\end{proof}

The following proposition can be proved in a similar manner to that of \cite[Proposition 3.1]{Wjde14} (or \cite[Lemma 3.6]{PZh}), so the proof will not be duplicated here.

\begin{proposition}\label{pr3.4}
Let $d$ and $C$ be positive constants.  For any given $g_0,R>0$, and any function $w_0\in C^2([0,g_0])$ satisfying $w_0'(0)=w_0(g_0)=0$ and $w_0>0$ in $[0,g_0)$, there exists $\bar\nu>0$  such that when $\nu>\bar\nu$, $(w,g)$ satisfies
\begin{eqnarray*}
\left\{\begin{array}{ll}
w_t-d\Delta w\geq Cw,&t>0,\ 0< r<g(t),\\[1mm]
w_r(t,0)=0=w(t,g(t)),&t>0,\\[1mm]
g'(t)=-\nu w_r(t,g(t)),&t>0,\\[1mm]
w(0,r)=w_0(r),\ g(0)=g_0,&0\leq r\leq g_0,
\end{array}\right.
\end{eqnarray*}
we must have $\liminf_{t\to\infty}g(t)>R$.
\end{proposition}

In order to derive the criteria for spreading and vanishing, we define $$\Sigma=\{\ell>0:
 \lambda_1(\ell;a_1,d_1)=0\ {\rm\ or}\ \lambda_1(\ell;a_2,d_2)=0\}.$$
According to the monotonicity of $\lambda_1(\ell;a_1,d_1)$ and $\lambda_1(\ell;a_2,d_2)$ with respect to $\ell$, we easily know that $\Sigma$ contains at most two elements. If one of the functions $a_1$ and $a_2$ satisfies  condition {\bf(A1)} or condition {\bf(A2)}, it follows from Proposition \ref{p3.2} that $\Sigma\not=\emptyset$.

In what follows we regard $h_0$ and $\mu$ as the varying parameters to study the  criteria for spreading and vanishing. Suppose that $\Sigma\neq\emptyset$ and $h^*=\min\Sigma\in(0,\infty)$, i.e., either $\lambda_1(h^*;a_1,d_1)=0$ or $\lambda_1(h^*;a_2,d_2)=0$.
\begin{lemma}\label{lem3.3}
{\rm (i)} If $h_\infty<\infty$, then $h_\infty\leq h^*$.

{\rm (ii)} If $h_0<h^*$, then there exist two positive numbers $\mu_0$, $\mu^0$ such that $h_\infty\leq h^*$ for any $0<\mu\leq\mu_0$ and $h_\infty=\infty$ for any $\mu\geq\mu^0$.
\end{lemma}
\begin{proof}
(i) Assume on the contrary that $h^*<h_\infty<\infty$. According to Proposition \ref{pr3.2}(ii) we derive either $\lambda_1(h_\infty;a_1,d_1)<0$ or $\lambda_1(h_\infty;a_2,d_2)<0$, which is in contradiction to the conclusion of Lemma \ref{lem3.1}. Therefore, $h_\infty\leq h^*$.

(ii) It follows from Lemma \ref{lem3.2} and part (i)  that when $h_0<h^*$, there is  a $\mu_0>0$ such that $h_\infty\leq h^*$ for $0<\mu\leq\mu_0$. In order to discuss the other assertion we consider the following auxiliary problem
\begin{eqnarray*}
\left\{\begin{array}{ll}
\underline u_t-d_1\Delta\underline u= -C\underline u,&t>0,\ 0\leq r<\underline h(t),\\[1mm]
\underline u_r(t,0)=0=\underline u(t,\underline h(t)),&t>0,\\[1mm]
\underline h'(t)=-\mu \underline u_r(t,\underline h(t)),&t>0,\\[1mm]
\underline u(0,r)=u_0(r),\ \underline h(0)= h_0,&0\leq r\leq  h_0,
\end{array}\right.
\end{eqnarray*}
where $C=\|a_1-b_1u-c_1v\|_\infty$. Utilizing Proposition \ref{pr3.4} we know that for any given $R>h^*$, there exists a $\mu^0>0$ such that $\underline h(\infty)>R$ for all $\mu\geq\mu^0$.  The comparison principle infers  $h(\infty)\geq \underline h(\infty)>h^*$. Hence the desired result follows from the discussion of part (i). The proof is ended.
\end{proof}
Finally, we present the sufficient conditions for spreading and vanishing, which is a principle theorem in this section.
\begin{theorem}\label{the3.2}
{\rm (i)} If $h_0\geq h^*$, then $h_\infty=\infty$ for all $\mu>0$.

{\rm (ii)} If $h_0< h^*$, then there exist  two positive numbers $\mu_*\leq\mu^*$ such that $h_\infty=\infty$ for any $\mu>\mu^*$, whereas $h_\infty\leq h^*$ for any $0<\mu\leq\mu_*$ or $\mu=\mu^*$.
\end{theorem}
\begin{proof}
(i) Due to the estimate (\ref{2.2}) it is easy to see that $h_\infty>h^*$ if $h_0\geq h^*$. Then it  follows from Lemma \ref{lem3.3}(i) that $h_\infty=\infty$ for all $\mu>0$.

(ii) The argument is essentially parallel to that of  Theorem 4.11 in \cite{KY}, but for completeness and the reader's convenience  we provide the details below.

Define $\mathcal{X}^*=\{\mu>0:\ h_\infty\leq h^*\}$. By means of Lemma \ref{lem3.3}(ii) we find that $(0,\mu_0]\subset\mathcal{X}^*$ and $\mathcal{X}^*\cap[\mu^0,\infty)=\emptyset$. Consequently, $\mu^*:=\sup\mathcal{X}^*\in[\mu_0,\mu^0]$. In view of this definition and Lemma \ref{lem3.3}(i) we know that $h_\infty=\infty$ when $\mu>\mu^*$. Therefore, $\mathcal{X}^*\subset(0,\mu^*]$.

We assert that $\mu^*\in\mathcal{X}^*$. Assume for contradiction that $h_\infty=\infty$ for $\mu=\mu^*$. Then we can select $T>0$ such that $h(T)>h^*$. In order to emphasize the dependence of the solution $(u,v,h)$ of (\ref{1.1}) on $\mu$, we now write $(u_\mu,v_\mu,h_\mu)$ instead of $(u,v,h)$. Hence we have $h_\mu(T)>h^*$. By virtue of the continuous dependence of $(u_\mu,v_\mu,h_\mu)$ on $\mu$, we can choose $\varepsilon>0$ small enough so that $h_\mu(T)>h^*$ for any $\mu\in[\mu^*-\varepsilon,\mu^*-\varepsilon]$. And then it can be derived that
$$\lim_{t\to\infty}h_\mu(t)>h_\mu(T)>h^*,\ \forall\ \mu\in[\mu^*-\varepsilon,\mu^*-\varepsilon],$$
which implies that $[\mu^*-\varepsilon,\mu^*-\varepsilon]\cap\mathcal{X}^*=\emptyset$. Thus we have $\sup\mathcal{X}^*\leq\mu^*-\varepsilon$, which is in contradiction to the definition of $\mu^*$. This proves our assertion that $\mu^*\in\mathcal{X}^*$.

Define $$\mathcal{X}_*=\{\kappa:\ \kappa\geq\mu_0\ {\rm such\ that}\ h_{\mu,\infty}\leq h^*\ {\rm for\ any}\ 0<\mu\leq\kappa\}$$
where $\mu_0$ is given in Lemma \ref{lem3.2}. Apparently, $\mu_*:=\sup\mathcal{X}_*\leq\mu^*$ and $(0,\mu_*)\subset\mathcal{X}_*$. Similar to the above, it can be shown that $\mu_*\in\mathcal{X}_*$. This completes the proof of Theorem \ref{the3.2}.
\end{proof}
\begin{remark}
For problem $(\ref{1.1})$, Theorem \ref{the3.2} does not provide any information for spreading success and spreading failure   when $\mu_*<\mu<\mu^*$. But for problem $(\ref{a1})$, we will give a threshold result for spreading-vanishing in section \ref{sec5}.
\end{remark}

\section{Long time behavior of $(u,v)$ for the spreading case} \label{sec4}
\setcounter{equation}{0}

The goal of this section is to deal with the long time behavior of $(u,v)$ for  spreading case: $h_\infty=\infty$. First of all we give the existence and uniqueness of positive solution to the logistic type elliptic equation
\begin{eqnarray}\label{4.1}
-d\Delta u=u(q(|x|)- p(|x|)u),& x\in\mathbb{R}^N,
\end{eqnarray}
Here, and in the following,  $d$ denotes given positive number, $q(|x|)$ and $p(|x|)$ are assigned functions in $C^\gamma([0,\infty))\cap L^\infty([0,\infty))$ with $0<\underline{p}\leq p(|x|)\leq \bar{p}<\infty\ {\rm in} \  [0,\infty).$

The following result is a special case of Theorem 7.12 in \cite{D}.
\begin{proposition}\label{pr4.1}
Suppose that there exist positive constants $\underline{q},$ $\bar{q}$ and $\rho\in(-2,0]$ such that
\begin{equation}\label{4.2}
{\underline{q}}=\liminf_{r\to\infty}\frac{q(r)}{r^\rho},\  \ \ {\bar{q}}=\limsup_{r\to\infty}\frac{q(r)}{r^\rho}.
\end{equation}
Then problem {\rm (\ref{4.1})} admits a unique positive  solution $u(r)$ satisfying
$$\underline{q}/{\bar{p}}\leq\liminf_{r\to\infty}\frac{u(r)}{r^\rho},\ \ \ \limsup_{r\to\infty}\frac{u(r)}{r^\rho}\leq \bar{q}/{\underline{p}}.$$
\end{proposition}

Let us point out that $(\ref{4.2})$ implies the assumption $({\bf A2})$. Besides, we denote the unique positive (radial) solution for problem $(\ref{4.1})$ by $\hat u(r)$, which will be used in the sequel.

\begin{theorem}\label{the4.2}
Suppose that $q(r)$  satisfies {\rm(\ref{4.2})}, and $\phi(r)\not\equiv0$ is a continuous, nonnegative and bounded function. Let $u(t,r)$ be the unique solution of the parabolic problem
\begin{eqnarray*}
\left\{\begin{array}{ll}
u_t-d\Delta u=u(q(r)-p(r) u),&t>0,\ 0\leq r<\infty,\\[1mm]
u_r(t,0)=0,&t\geq0,\\[1mm]
u(0,r)=\phi(r),&0\leq r<\infty.
\end{array}
\right.
\end{eqnarray*}
Then $\lim_{t\to\infty}u(t,r)=\hat u(r)$ uniformly on any compact subset of $[0,\infty)$.
\end{theorem}

\begin{proof}
 Taking account of Proposition \ref{p3.2}, the condition (\ref{4.2}) results in $\lambda_1(\infty;q,d)<0$. Consequently, for $\ell\gg1$, the elliptic problem
\begin{eqnarray}\label{4.3}
\left\{\begin{array}{ll}
-d\Delta u=u(q(|x|)- p(|x|)u),& x\in B_\ell,\\[1mm]
u(|x|)=0,&x\in \partial B_\ell
\end{array}\right.
\end{eqnarray}
admits a unique positive (radial) solution, denoted by $\hat u_\ell(r)$.

From the positivity of parabolic equations it follows that $u(t,r)>0$ for  all $t>0$ and $r\geq0$. Thus we may suppose that $\phi(r)>0$ for all $r\geq0$. For $\ell\gg1$, let $u_\ell(t,r)$ be the unique solution of the initial-boundary value problem
\begin{eqnarray}\label{4.4}
\left\{\begin{array}{ll}
u_t-d\Delta u=u(q(r)- p(r)u),&t>0,\ 0\leq r<\ell,\\[1mm]
u_r(t,0)=0=u(\ell),&t\geq0,\\[1mm]
u(0,r)=\phi(r),&0\leq r\leq\ell.
\end{array}
\right.
\end{eqnarray}
By the comparison principle we derive
\begin{equation}\label{4.5}
u(t,r)\geq u_\ell(t,r), \ \ \forall\ t\geq0,\ 0\leq r\leq\ell.
\end{equation}

 Let $\psi(r)$ denote the  positive eigenfunction corresponding to $\lambda_1(\ell;q,d)<0$.  Then it is not hard to verify that  $\delta_1\psi(r)$ is a lower solution of problem (\ref{4.3}) if $\delta_1$ is a sufficiently small positive number. According to the above analysis on $\phi(r)$, we know that there exists a sufficiently small $\delta_2>0$ so that $\delta_2\psi(r)\leq\phi(r)$ on $[0,\ell]$. If  $\delta=\min\{\delta_1,\delta_2\}$, then $\delta\psi(r)\leq\phi(r)$ on $[0,\ell]$ and is a lower solution of (\ref{4.3}). On the other hand, it is evident that a suitably large $C>0$ is an upper solution of (\ref{4.3}). Let $\bar u_\ell(t,r)$ and $\underline u_\ell(t,r)$ be the unique solution of problem (\ref{4.4}) with $\phi(r)=C$ and $\phi(r)=\delta\psi(r)$, respectively. With the aid of the comparison principle one can derive
\begin{equation}\label{4.6}
\bar u_\ell(t,r)\geq u_\ell(t,r)\geq \underline u_\ell(t,r),\ \ \forall\ t\geq0,\ 0\leq r\leq\ell,
\end{equation}
and $\bar u_\ell(t,r)$ is decreasing and $\underline u_\ell(t,r)$ is increasing with respect to $t$. Furthermore, it follows that both $\lim_{t\to\infty}\bar u_\ell(t,r)=\bar u_\ell(r)$ and $\lim_{t\to\infty}\underline u_\ell(t,r)=\underline u_\ell(r)$ are positive solutions of (\ref{4.3}). Due to the uniqueness one can achieve $\bar u_\ell(r)\equiv\underline u_\ell(r)\equiv\hat u_\ell(r)$. Combining this with (\ref{4.6}) and (\ref{4.5}) gives
\begin{equation}\label{4.7}
\liminf_{t\to\infty} u(t,r)\geq \hat u_\ell(r)\ \ {\rm uniformly\ on\ }[0,\ell].
\end{equation}
 In addition, by virtue of the regularity theory and compactness argument, we can deduce that $\hat u_\ell(r)\to\hat u(r)$ in $C^{2+\gamma}_{\rm loc}([0,\infty))$ as $\ell\to\infty$. Therefore, it follows from (\ref{4.7}) that
\begin{equation}\label{4.8}
\liminf_{t\to\infty}u(t,r)\geq\hat u(r)\ \ {\rm uniformly\ on\ any\ compact\ subset\ of}\  [0,\infty).
\end{equation}

Let $C=\max\{\|\phi\|_\infty,\|q\|_\infty/\underline{p}\}$ and $u_C(t,r)$ be the unique solution of
\begin{eqnarray*}
\left\{\begin{array}{ll}
u_t-d\Delta u=u(q(r)- p(r)u),&t>0,\ 0\leq r<\infty,\\[1mm]
u_r(t,0)=0,&t\geq0,\\[1mm]
u(0,r)=C,&0\leq r<\infty.
\end{array}
\right.
\end{eqnarray*}
Then it is not too difficult to obtain that $u_C(t,r)$ is decreasing with respect to $t$, $u_C(t,r)\geq u(t,r)$ for all $t\geq0$ and $r\geq0$, $u_C(t,r)\geq\hat u(r)$ for all $t\geq0$ and $r\geq0$ because of $\hat u(r)<C$, and $\lim_{t\to\infty}u_C(t,r)=\tilde u(r)$ uniformly on any compact subset of $[0,\infty)$, where $\tilde u(r)$ is some positive solution of (\ref{4.1}). By the uniqueness of solution for problem (\ref{4.1}), we easily derive that
\begin{equation}\label{4.9}
\limsup_{t\to\infty}u(t,r)\leq\hat u(r)\ \ {\rm uniformly\ on\ any\ compact\ subset\ of}\  [0,\infty).
\end{equation}
The desired result immediately follows from (\ref{4.8}) and (\ref{4.9}). The proof is finished.
\end{proof}

Let $q_i(r)\in C^\gamma([0,\infty))$, $i=1,\ 2$. Assume that there exist $\rho\in(-2,0]$, and positive constants $\underline{q}_i$ and $\bar{q}_i$, such that $$
\underline{q}_i=\liminf_{r\to\infty}\frac{ q_i(r)}{r^\rho},\ \ \  \bar{q}_i=\limsup_{r\to\infty}\frac{q_i(r)}{r^\rho},$$
By means of Proposition \ref{pr4.1}, problem (\ref{4.1}) with $q(r)$ replaced by $ q_i(r)$  admits a unique positive solution, denoted by $u_i(r)$, which satisfies
\begin{equation}\label{4.10}
\underline{q}_i/{\bar{p}}\leq\liminf_{r\to\infty}\frac{u_i(r)}{r^\rho},\ \ \ \limsup_{r\to\infty}\frac{u_i(r)}{r^\rho}\leq \bar{q}_i/{\underline{p}}.
\end{equation}
\begin{proposition}\label{pr4.2}
If $q_i(r)$ $(i=1,2)$  satisfies the above assumptions, and $q_1(r)\leq q_2(r)$ for all $r\geq0$, then
$$u_1(r)\leq u_2(r),\ \ \forall\ r\geq0.$$
\end{proposition}

\begin{proof} For convenience, we denote $g(x)=g(r)$, where $g$ may be one of the functions $p,q_i,u_i, i=1,2$. The proof will be divided into three steps.

{\it Step 1.} It will be shown that there exists $L>1$ large enough so that, if $|x_*|>L$ and $u_1(x_*)>m_*u_2(x_*)$ for some $m_*\geq m>1$, then we can find $y_*\in \mathbb{R}^N$, and positive constants $c_0=c_0(L, m)$ and $\delta_0=\delta_0(L,m)$ independent of $x_*$ and $m_*$, such that
$$|y_*-x_*|=\delta_0x_*^{-\rho/2},\ \ \ u_1(y_*)>(1+c_0)m_*u_2(y_*).$$

By virtue of the assumptions for $q_i(x)$ and (\ref{4.10}), it is easy to see that for all large $L>1$ and $|x|>L$,
\begin{equation}\label{4.11}
\frac{\underline{q}_i} 2|x|^\rho<q_i(x)<2 \bar{q}_i|x|^\rho,\ \  \frac12\underline{p}<p(x)<2\bar{p},\ \ \frac{\underline{q}_i}{2\bar{p}}{|x|^\rho}<{u_i(x)} <\frac{2\bar{q}_i}{\underline{p}}{|x|^\rho},\ \ i=1,2.
\end{equation}
We now fix $L>1$ large enough such that $L^{-1-\rho/2}<1/2$ and (\ref{4.11}) holds for any $|x|>L/2$. Define
$$\Omega_0=\{x\in\mathbb{R}^N:\ u_1(x)>m_*u_2(x)\}\cap B_\delta(x_*),$$
where $\delta=\delta_0x_*^{-\rho/2}$, $B_\delta(x_*)=\{x\in\mathbb{R}^N:\ |x-x_*|<\delta\}$, and $\delta_0\in(0,1)$ is to be determined later. On account of $|x_*|>L$ and our choice of $L$, $x\in \Omega_0$ implies
\begin{equation}\label{4.12}
|x_*|/2<|x|<3|x_*|/2.
\end{equation}

Next, we consider $u_1(x)-m_*u_2(x)$ in $\Omega_0$.  By virtue of (\ref{4.11}), (\ref{4.12}) and the assumption that $u_1(x)>m_*u_2(x)$ in $\Omega_0$, it can be deduced that for $x\in \Omega_0$,
  \begin{eqnarray*}
-d\Delta(u_1-m_*u_2)&=&q_1(x)u_1- p(x)u_1^2-m_*(q_2(x)u_2- p(x)u_2^2) \\[1mm]
&\leq&  q_2(x)(u_1-m_*u_2)-p(x)(m_*^2u_2^2-m_*u_2^2)\\[1mm]
&\leq&2\bar{q}_2|x|^\rho(u_1-m_*u_2)-(\underline{p}/8\bar{p})
\underline{q}_2^2m_*(m_*-1)|x|^{2\rho}\\[1mm]
&\leq& 2^{1-\rho} \bar{q}_2 |x_*|^\rho(u_1-m_*u_2)-(\underline{p}/8\bar{p})\underline{q}_2^2
(3/2)^{2\rho}m_*(m_*-1)|x_*|^{2\rho}\\[1mm]
   &\leq& C_*|x_*|^\rho(u_1-m_*u_2)-c_*m_*|x_*|^{2\rho},
\end{eqnarray*}
where $C_*=2^{1-\rho} \bar{q}_2$ and  $c_*=(\underline{p}/8\bar{p})\underline{q}_2^2(3/2)^{2\rho} (m-1)$.

Now we define
$$w(x)=(2dN)^{-1}c_*m_*|x_*|^{2\rho}(\delta^2-|x-x_*|^2).$$
Obviously, $w(x)>0$ in $B_\delta(x_*)$ and $-\Delta w(x)=d^{-1}c_*m_*|x_*|^{2\rho}$. It follows that for $x\in \Omega_0$,
\begin{equation}\label{4.13}
-d\Delta(u_1-m_*u_2+w)\leq C_*|x_*|^\rho(u_1-m_*u_2)\leq C_*|x_*|^\rho(u_1-m_*u_2+w).
\end{equation}

Let $\lambda_1(\Omega)$ denote the first eigenvalue of $-\Delta$ over $\Omega$ under homogeneous Dirichlet boundary conditions. Then
$$\lambda_1(\Omega_0)\geq\lambda_1(B_\delta(x_*))=\delta^{-2}\lambda_1(B_1(x_*)).$$
By use of $\delta=\delta_0|x_*|^{-\rho/2}$, it is easy to get that  $\lambda_1(\Omega_0)\geq \delta_0^{-2}|x_*|^{\rho}\lambda_1(B_1(x_*))$. Note that $\lambda_1(B_1(x_*))$ is independent of $x_*$. We now select $\delta_0\in(0,1)$ small enough such that $\delta_0^{-2}\lambda_1(B_1(x_*))>C_*$ and hence $\lambda_1(\Omega_0)>C_*|x_*|^{\rho}$. Then making use of the maximum principle (Theorem 2.8 in \cite{BNV}) and (\ref{4.13}), we derive
$$u_1(x_*)-m_*u_2(x_*)+w(x_*)\leq\max_{\partial \Omega_0}(u_1-m_*u_2+w).$$
It can be seen that the maximum of $(u_1-m_*u_2+w)$ over $\partial \Omega_0$ has to be achieved by some $y_*\in \partial B_\delta(x_*)$ since $y\in\partial \Omega_0\setminus\partial B_\delta(x_*)$ satisfies, by the definition of $\Omega_0$, $u_1(y)=m_*u_2(y)$ and hence
$$u_1(y)-m_*u_2(y)+w(y)=w(y)\leq w(x_*)<u_1(x_*)-m_*u_2(x_*)+w(x_*).$$
Consequently we can take $y_*\in\partial \Omega_0$ satisfying $|y_*-x_*|=\delta$ (thus $w(y_*)=0$) so that
\begin{eqnarray*}
u_1(y_*)-m_*u_2(y_*)&=&u_1(y_*)-m_*u_2(y_*)+w(y_*)\\[1mm]
&\geq&u_1(x_*)-m_*u_2(x_*)+w(x_*)\\[1mm]
&>&w(x_*)=(2dN)^{-1}c_*m_*|x_*|^{2\rho}\delta^2\\[1mm]
&=&(2dN)^{-1}c_*m_*\delta_0^2|x_*|^{\rho}\geq c_1m_*y_*^\rho,
\end{eqnarray*}
where $c_1=(2dN)^{-1}c_*\delta_0^22^\rho>0$, and (\ref{4.12}) has been used. By means of (\ref{4.11}), it follows that
$$u_1(y_*)-m_*u_2(y_*)\geq c_1m_*y_*^\rho\geq c_1(\underline{p}/2\bar{q}_2)m_*u_2(y_*).$$
Therefore we can choose $c_0=c_1(\underline{p}/2\bar{q}_2)$ and obtain our desired results.

{\it Step 2.} We  show that $u_1(x)\leq u_2(x)$ for  sufficiently large $|x|>0$.
Let
$$m_0=\inf\{m>0:\ u_1(x)\leq mu_2(x),\ \forall\ |x|\gg1\},\ i.e.,\ m_0=\limsup_{|x|\to\infty}\frac{u_1(x)}{u_2(x)}.$$
In view of (\ref{4.10}) we know that $m_0$ is finite. If $m_0\leq1$, then $u_1(x)\leq u_2(x)$ for  $|x|\gg1$.

Suppose by way of contradiction that $m_0>1$. Then there exist a constant $\tilde m\in(1,m_0)$ and a sequence $\{x_n\}$, with $|x_n|\to \infty$, such that
$$\frac{u_1(x_n)}{u_2(x_n)}>\tilde m, \ \ \ {\rm for\ }n=1,2,\cdots.$$
On the other hand, we can find an integer $j>1$ such that
$$(1+c_0)^j\tilde m>\sup_{|x|>L}\frac{u_1(x)}{u_2(x)}.$$
Since $|x_n|\to\infty$, there exists $n_0$ large enough so that
$|x_{n_0}|(1/2)^j>L$. Taking $x_*=x_{n_0}$ and $m_*=\tilde m$ in {\it Step 1}, we can find $y_*=y_1$ such that $$|y_1-x_*|=\delta_0|x_*|^{-\rho/2},\ \ u_1(y_1)>(1+c_0)\tilde mu_2(y_1).$$
Thanks to $L^{-1-(\rho/2)}<1/2$, it follows that
$|y_1|\geq |x_*|-\delta_0|x_*|^{-\rho/2}\geq |x_{n_0}|(1-L^{-1-\rho/2})>L.$
We now take $x_*=y_1$ and $m_*=(1+c_0)\tilde m$, and then can pick $y_2$ such that
$$|y_2-y_1|=\delta_0|y_1|^{-\rho/2},\ \ u_1(y_2)>(1+c_0)^2\tilde mu_2(y_2).$$
Moreover, $|y_2|\geq\frac12 |y_1|\geq(\frac12)^2|x_{n_0}|>L$.

Repeating this procedure, we must derive $y_j$ satisfying
$$u_1(y_j)>(1+c_0)^j\tilde mu_2(y_j),\ \ |y_j|\geq(\frac 12)^j |x_{n_0}|>L.$$
As a consequence
$$\frac{u_1(y_j)}{u_2(y_j)}\geq (1+c_0)^j\tilde m>\sup_{|x|>L}\frac{u_1(x)}{u_2(x)}.$$
This is a contradiction.

{\it Step 3.} It follows from {\it Step 2} that there exists $\ell>0$ sufficiently large such that $u_1(x)\leq u_2(x)$ for all $|x|>\ell$. Since $q_1(x)\leq q_2(x)$ for all $x\in\mathbb{R}^N$, it is clear that for any $R>\ell$,
$$-d\Delta u_1-q_1(|x|)u_1+ p(|x|)u_1^2=0\leq-d\Delta u_2-q_1(|x|)u_2+ p(|x|)u_2^2, \ x\in B_R(0),$$
and $\limsup_{r\to R}(u_1^2-u_2^2)\leq0.$ By means of the comparison principle (Lemma 5.6 in \cite{D}) we  can derive that  $u_1(x)\leq u_2(x)$  for all $x\in B_R(0)$.
The proof is finished.
\end{proof}

\begin{theorem}\label{theor4.2}
Assume that there exist $\rho\in(-2,0]$, and  positive numbers $\underline{a}_i,$ $\bar{a}_i$
 such that
$$
\underline{a}_i=\liminf_{r\to\infty}\frac{a_i(r)}{r^\rho},\ \ \ \bar{a}_i=\limsup_{r\to\infty}\frac{a_i(r)}{r^\rho}
$$
and $$\underline{a}_2\underline{b}_1-\bar{a}_1\bar{b}_2>0, \ \ \  \underline{a}_1\underline{c}_2-\bar{a}_2\bar{c}_1>0.$$
Then the problem
\begin{equation}\label{4.14}
\left\{\begin{array}{lll}
-d_1\Delta u=u(a_1(r)-b_1(r)u-c_1(r)v) & {\rm in \ }\mathbb{R}^N,\\[1mm]
-d_2\Delta v=v(a_2(r)-b_2(r)u-c_2(r)v)&{\rm in \ }\mathbb{R}^N
\end{array}\right.
\end{equation}
admits a positive solution. Furthermore, any positive solution  $(u,v)$ of $(\ref{4.14})$ fulfills
\begin{equation}\label{4.15}
\bar u(r)\geq u(r)\geq\underline u(r),\ \ \bar v(r)\geq v(r)\geq\underline v(r), \ \forall \ r\geq0,
\end{equation}
where $\bar u$, $\bar v$, $\underline u$ and $\underline v$ will be given in the following proof.
\end{theorem}
\begin{proof}
{\it Step 1.} The construction of $\bar u$, $\bar v$, $\underline u$ and $\underline v$.

By Proposition \ref{pr4.1}, the problem
\begin{eqnarray}\label{4.16}
 -d_1\Delta u=u(a_1(r)-b_1(r)u)\ \ {\rm in \ }\mathbb{R}^N
\end{eqnarray}
has a unique positive solution, denoted by $\bar u(r)$, satisfying $$\frac{\underline{a}_1}{\bar{b}_1}\leq\liminf_{r\to\infty}\frac{\bar u(r)}{r^\rho},\ \ \ \limsup_{r\to\infty}\frac{\bar u(r)}{r^\rho}\leq \frac{\bar{a}_1}{\underline{b}_1}. $$
Since $\underline{a}_2\underline{b}_1-\bar{a}_1\bar{b}_2>0$, we have $$\liminf_{r\to\infty}\frac{a_2(r)-b_2(r)\bar u(r)}{r^\rho}>0.$$
Again making use of Proposition \ref{pr4.1} we know that the problem
\begin{eqnarray}\label{4.17}
 -d_2\Delta v=v(a_2(r)-b_2(r)\bar u-c_2(r)v) \ \ {\rm in \ }\mathbb{R}^N
\end{eqnarray}
admits a unique positive solution, denoted by $\underline v(r)$. Similar to the above we easily see that the problems
\begin{eqnarray}\label{4.18}
 -d_2\Delta v=v(a_2(r)-c_2(r)v) \ \ {\rm in \ }\mathbb{R}^N
\end{eqnarray}
and
\begin{eqnarray}\label{4.19}
 -d_1\Delta u=u(a_1(r)-b_1(r) u-c_1(r)\bar v) \ \ {\rm in \ }\mathbb{R}^N
\end{eqnarray}
possess unique positive solutions $\bar v(r)$ and  $\underline u(r)$, respectively.

Applying the comparison principle (Proposition \ref{pr4.2}) asserts that $\underline u(r)\leq \bar u(r),\ \underline v(r)\leq \bar v(r)$ for all $r\geq0$.

{\it Step 2.} Existence of positive solution to problem (\ref{4.14})

We know from {\it Step 1} that $\underline u,\ \underline v,\ \bar u$ and $\bar v$ are the coupled ordered lower and upper solutions of problem (\ref{4.14}). For any given $\ell>0$, it is evident that
$\underline u,\ \underline v,\ \bar u$ and $\bar v$ are also the coupled ordered lower and upper solutions of the problem
\begin{equation}\label{4.20}
\left\{\begin{array}{lll}
-d_1\Delta u=u(a_1(r)-b_1(r)u-c_1(r)v) &{\rm in \ }B_\ell,\\[1mm]
-d_2\Delta v=v(a_2(r)-b_2(r)u-c_2(r)v)&{\rm in \ }B_\ell,\\[1mm]
 u(\ell)=\bar u(\ell),\ v(\ell)=\underline v(\ell).
\end{array}\right.
\end{equation}
With the aid of the standard upper and lower solutions argument we conclude that problem (\ref{4.20}) admits at least one positive solution, denoted by $(u_\ell, v_\ell)$, satisfying
$$\underline u(r)\leq u_\ell(r)\leq \bar u(r),\ \ \underline v(r)\leq v_\ell(r)\leq \bar v(r),\ \forall\ 0\leq r\leq\ell.$$
Taking advantage of the regularity theory and compactness argument infers that there exists a pair of $(u,v)$ such that $(u_\ell,v_\ell)\to (u,v)$ in $[C^2_{\rm loc}([0,\infty))]^2$ as $\ell\to\infty$ and $(u,v)$ solves (\ref{4.14}).

By virtue of Proposition \ref{pr4.2}, it can be deduced that any positive solution  $(u,v)$ of problem $(\ref{4.14})$ fulfills (\ref{4.15}). The proof is ended.
\end{proof}
Employing the comparison principle (Proposition \ref{pr4.2}),  regularity theory and compactness argument, we can demonstrate the following conclusion.
\begin{proposition}\label{pr4.4}
Suppose that $0<\epsilon\ll1$ and  $q(r)$ satisfies $(\ref{4.2})$. If   $u_\varepsilon^{\pm}(r)$ is the unique positive solution of problem
\begin{eqnarray*}
-d\Delta u=u(q(r)\pm \varepsilon r^{\rho}- p(r)u)\ \ {\rm in \ }\mathbb{R}^N,
\end{eqnarray*} then $\lim_{\varepsilon\to0}u_\varepsilon^{\pm}(r)=\hat u(r)$ uniformly on any compact subset of $[0,\infty)$.
\end{proposition}

\begin{theorem}\label{the4.1}
Let $a_1(r)$ and $a_2(r)$ be as in Theorem $\ref{theor4.2}$.
If  $h(\infty)=\infty$ and $(u(t,r), v(t,r),h(t))$ is the solution of problem {\rm(\ref{1.1})}, then the following inequalities hold uniformly on any compact subset of $[0,\infty)$,
\begin{eqnarray}
\underline  u(r)\leq\liminf_{t\to\infty}u(t,r),&&\ \limsup_{t\to\infty}u(t,r)\leq\bar u(r),\label{4.21}\\[1mm]
\underline v(r)\leq\liminf_{t\to\infty}v(t,r),&&\ \limsup_{t\to\infty}v(t,r)\leq\bar v(r).\label{4.22}
\end{eqnarray}
Here $\underline u(r)$, $\bar u(r)$, $\underline v(r)$ and $\bar v(r)$ are given in the proof of Theorem $\ref{theor4.2}$.
\end{theorem}
\begin{proof}
{\it Step 1.} Let $\tilde u(t,r)$ be the only positive solution for $u_t-d_1\Delta u=u(a_1(r)-b_1(r) u)$, $t>0,\ 0\leq r< \infty$ with the boundary condition $u_r(t,0)=0$ for $t>0$ and the initial data
\begin{eqnarray*}
u(0,r)=\left\{\begin{array}{ll}
u_0(r),&0\leq r\leq h_0,\\[1mm]
0,&r> h_0.
\end{array}
\right.
\end{eqnarray*}
By means of the comparison principle we obtain $u(t,r)\leq\tilde u(t,r)$ for all $t\geq0$ and $0\leq r\leq h(t)$. In light of Theorem \ref{the4.2}, $\lim_{t\to\infty}\tilde u(t,r)=\bar u(r)$ uniformly on any compact subset of $[0,\infty)$, where $\bar u(r)$ is the only positive solution of problem (\ref{4.16}). Due to $h(\infty)=\infty$, we easily get the second inequality of (\ref{4.21}).

Similar to the above, we can derive $\limsup_{t\to\infty}v(t,r)\leq\bar v(r)$ uniformly on any compact subset of $[0,\infty)$, where $\bar v(r)$ is the only positive solution for problem (\ref{4.18}).

{\it Step 2.} In this step we shall show the remaining two inequalities of this theorem.

For any given $0<\varepsilon\ll1$ and $\ell\gg1$, we can choose $T$ sufficiently large so that
$$h(t)>\ell,\ \ u(t,r)<\bar u(r)+\varepsilon r^{\rho},\ \forall\ t\geq T,\ 0\leq r<\ell.$$
In addition, for $\ell\gg1$, the elliptic problem
\begin{eqnarray}\label{4.23}
\left\{\begin{array}{ll}
-d_2\Delta v=v(a_2(r)-c_2(r)v-b_2(r)(\bar u(r)+\varepsilon r^{\rho}))&{\rm in \ }B_\ell,\\[1mm]
v(\ell)=0
\end{array}\right.
\end{eqnarray}
admits a unique positive solution, denoted by $v_\ell^0(r)$.
Since $v(T,r)>0$ for any $0\leq r\leq\ell$, there exists a positive  number $\delta<1$ so that $v(T,r)\geq\delta v_\ell^0(r)$ for any $0\leq r\leq\ell$. It is easy to verify that $\delta v_\ell^0(r)$ is a lower solution of problem (\ref{4.23}).

Let $v_\ell^0(t,r)$ be the unique solution of the following parabolic problem with fix boundary condition
\begin{eqnarray}\label{4.24}
\left\{\begin{array}{ll}
v_t-d_2\Delta v=v(a_2(r)-c_2(r)v-b_2(r)(\bar u(r)+\varepsilon r^{\rho})),&t>T,\ 0\leq r<\ell,\\[1mm]
v_r(t,0)=0=v(t,\ell),&t\geq T,\\[1mm]
v(T,r)=\delta v_\ell^0(r),&0\leq r\leq \ell.
\end{array}
\right.
\end{eqnarray}
The comparison principle asserts $v(t,r)\geq v_\ell^0(t,r)$ for any $t\geq T$ and $0\leq r\leq\ell$, and $v_\ell^0(t,r)$ is increasing with respect to $t$. Note that $C=\max\{K, \|a_2(r)\|_\infty/\underline{c}_2\}$ is an upper solution of problem (\ref{4.24}), where $K$ is established in Theorem \ref{t2.1}. Thus it can be deduced that $\lim_{t\to\infty}v_\ell^0(t,r)=v_\ell^0(r)$ uniformly on $[0,\ell]$. We further derive
\begin{equation}\label{4.25}
\liminf_{t\to\infty}v(t,r)\geq v_\ell^0(r)\ \ {\rm uniformly\  on}\ [0,\ell].
\end{equation}

Let $v_\varepsilon(r)$ be the unique positive solution of
\begin{eqnarray}\label{4.26}
-d_2\Delta v=v(a_2(r)-b_2(r)(\bar u(r)+\varepsilon r^{\rho})-c_2(r)v)\ \ {\rm in \ }\mathbb{R}^N.
\end{eqnarray}
By virtue of the comparison principle we obtain $v_\ell^0(r)\leq v_\varepsilon(r)$ on $[0,\ell]$ and $v_\ell^0(r)$ is increasing with respect to $\ell$. Making use of the regularity theory and compactness argument, it can be deduced that there exists a positive function $\tilde v(r)$ so that $v_\ell^0(r)\to\tilde  v(r)$ in $C^2_{\rm loc}([0,\infty))$ as $\ell\to\infty$, and $\tilde v(r)$ satisfies (\ref{4.26}). Owing to the uniqueness of solution we have $\tilde v(r)=v_\varepsilon(r)$, and so
\begin{equation}\label{4.27}
\lim_{\ell\to\infty}v_\ell^0(r)=v_\varepsilon(r)\ \ {\rm uniformly\ on\ any\ compact\ subset\ of}\  [0,\infty).
\end{equation}
In terms of Proposition \ref{pr4.4} we derive
\begin{equation}\label{d14}
\lim_{\varepsilon\to0}v_\varepsilon(r)=\underline v(r)\ \ {\rm uniformly\ on\ any\ compact\ subset\ of}\  [0,\infty),
\end{equation}
where $\underline v(r)$ is the only positive solution of problem (\ref{4.17}).

Now the first inequality of (\ref{4.22}) immediately follows from (\ref{4.25}), (\ref{4.27}) and (\ref{d14}). Similarly we can show
 $$\underline  u(r)\leq\liminf_{t\to\infty}u(t,r)\ \ {\rm uniformly\ on\ any\ compact\ subset\ of}\  [0,\infty),$$
where $\underline u(r)$ denotes the unique positive solution of problem (\ref{4.19}).
This completes the proof of Theorem \ref{the4.1}.
\end{proof}

\section{Corresponding results for problem (\ref{a1})}\label{sec5}
\setcounter{equation}{0}
In this section, we explain how the techniques developed for treating (\ref{1.1}) can be modified to derive similar results for (\ref{a1}).

We start with the counterpart of Theorem \ref{t2.1}.
\begin{theorem}\label{t5.1}
For any given $u_0$ and $v_0$ satisfying {\rm(\ref{a2})}, problem {\rm(\ref{a1})} has a unique global solution $(u,v,h)$, and
$$
(u,v,h)\in C^{1+\frac{\gamma}{2},2+\gamma}( Q)\times C^{1+\frac{\gamma}{2},2+\gamma}( D)\times C^{1+\frac{1+\gamma}{ 2}}(0,\infty),
$$
where $Q=\{(t,r)\in \mathbb{R}^2:t\in(0,\infty),\ r\in[0,h(t)]\}$ and $D=\{(t,r)\in \mathbb{R}^2:t\in(0,\infty),\ r\in[0,\infty)\}$. Moreover, there exist  positive constants $K$ and $K_1$ dependent on $d_i,\mu,\beta, \underline{b}_i, \bar b_i, \underline{c}_i, \bar{c}_i, h_0$ and $\|a_i,u_0,v_0\|_\infty$ such that
$$
0<u(t,r)\leq K, \ \forall\  t>0,\ 0\leq r<h(t); \ 0<v(t,r)\leq K,\ \forall\  t>0,\ 0\leq r<\infty;
$$
$$
\ 0<h'(t)\leq K,\ \forall\  t>0;\ \  \|u(t,\cdot)\|_{C^1[0,h(t)]}\leq K_1,\ \forall \ t\geq1;\ \ \|h'\|_{C^{\frac{\gamma}{2}}([n+1,n+3])}\leq K_1, \ \forall\ n\geq0.
$$
\end{theorem}
\begin{proof}
The proof is essentially the same as that of Theorem \ref{t2.1}. We merely point out two  modifications here. Firstly, similar to the argument for Theorem 2.1 in \cite{DL2} instead of Lemma 2.1 in \cite{GW}, it can be shown that for any $\nu\in(0,1)$, there is a $\tau>0$ such that problem (\ref{a1}) admits a unique local solution $(u,v,h)\in C^{\frac{1+\nu}{2},1+\nu}( Q_\tau)\times C^{\frac{1+\nu}{2},1+\nu}(D_\tau)\times C^{1+\frac{\nu}{2}}(0,\tau)$, where $Q_\tau=\{(t,r):t\in(0,\tau],\ r\in[0,h(t)]\}$ and $D_\tau=\{(t,r):t\in(0,\tau],\ r\in[0,\infty)\}$. Secondly, the regularity $v\in C^{1+\frac{\gamma}2,\,2+\gamma}((0,\tau]\times[0,\infty))$ comes from  the fact that $v\in C^{1+\frac{\gamma}2,\,2+\gamma}\big((0,\tau]\times[m,m+1]\big)$ for any $m\geq0$.
\end{proof}
In the remainder of this section, it is always assumed that there exists positive constants $\underline{a}_i$, $\bar{a}_i$ $(i=1,2)$ such that
\begin{equation}\label{e0}
\underline{a}_i=\liminf_{r\to\infty}\frac{a_i(r)}{r^\rho},\ \ \ \bar{a}_i=\limsup_{r\to\infty}\frac{a_i(r)}{r^\rho}.
\end{equation}

We next establish the spreading-vanishing dichotomy. To do this, we first exhibit the asymptotic behavior of $(u,v)$ for vanishing situation.
\begin{theorem}\label{the5.1}
Let $(u,v,h)$ be the solution of problem $(\ref{a1})$. If $h_\infty<\infty$, then
\begin{equation}\label{e1}\lim_{t\to\infty}\|u(t,\cdot)\|_{C([0,h(t)])}=0\end{equation}
and
\begin{equation}\label{e2}
\lim_{t\to\infty}v(t,r)=V(r)\ \ {\rm uniformly\ on\ any\ compact\ subset\ of}\  [0,\infty),
\end{equation}
where $V(r)$ is the only positive solution of the problem
\begin{eqnarray}\label{e3}
-d_2\Delta v=v(a_2(r)- c_2(r)v)\ \ {\rm in \ }\mathbb{R}^N
\end{eqnarray}
This result shows that if a new competitor can not penetrate deep into the habitat of  a well established native species, it will dies out eventually.
\end{theorem}
\begin{proof}
The limit (\ref{e1}) can be obtained analogously to the proof of Theorem \ref{the3.1}.

Set $C=\|v_0\|_\infty+\|a_2(r)\|_\infty/\underline{c}_2$ and let $v_C(t,r)$ be the only positive solution of
\begin{eqnarray*}
\left\{\begin{array}{ll}
v_t-d_2\Delta v=v(a_2(r)- c_2(r)v),&t>0,\ 0<r<\infty,\\[1mm]
v_r(t,0)=0,&t>0,\\[1mm]
v(0,r)=C,&0\leq r<\infty.
\end{array}
\right.
\end{eqnarray*}
Then $v(t,r)\leq v_C(t,r)$ for all $t\geq0$ and $r\geq0$ and $v_C(t,r)$ is decreasing with respect to $t$. Because $V(r)$ is the only positive solution of (\ref{e3}), by the standard method we can show that $\lim_{t\to\infty}v_C(t,r)=V(r)$ uniformly on any compact subset of $[0,\infty)$. Thus
\begin{eqnarray*}
\limsup_{t\to\infty}v(t,r)\leq V(r)\ \ {\rm uniformly\ on\ any\ compact\ subset\ of}\  [0,\infty).
\end{eqnarray*}
Since $u(t,r)\equiv0$ for any $t\geq0$ and $r\geq h(t)$, we can select sufficiently large $T>0$ such that for any $0<\varepsilon\ll1$,
$$u(t,r)<\varepsilon r^\rho, \ \ {\rm \forall \ t\geq T},\ r\geq0.$$
Similar to the discussion of {\it Step 2} in Theorem \ref{the4.1}, it is not difficult to derive that
\begin{eqnarray*}
\liminf_{t\to\infty}v(t,r)\geq V(r)\ \ {\rm uniformly\ on\ any\ compact\ subset\ of}\  [0,\infty).
\end{eqnarray*}
Therefore, (\ref{e2}) is verified. This completes the proof.
\end{proof}

Assume in the sequel that
\vspace{-1mm}\begin{quote}
 {\bf(A3)}\,there exists $0<R^*<\infty$ so that $\lambda_1(R^*;a_1(r)-c_1(r)V(r),d_1)=0$.
 \vspace{-1mm}\end{quote}
Let us point out that this assumption can hold when $\underline{a}_1\underline{c}_2-\bar{a}_2\bar{c}_1>0$.
Actually, it follows from Proposition \ref{pr4.1} that $V(r)$ satisfies
  $${\underline{a}_2}/{\bar{c}_2}\leq\liminf_{r\to\infty}\frac{V(r)}{r^\rho}\leq \limsup_{r\to\infty}\frac{V(r)}{r^\rho}\leq {\bar{a}_2}/{\underline{c}_2}. $$
Therefore, if $\underline{a}_1\underline{c}_2-\bar{a}_2\bar{c}_1>0$, then $\liminf_{r\to\infty}\frac{a_1(r)-c_1(r)V(r)}{r^\rho}>0$, which implies that the function $a_1(r)-c_1(r)V(r)$ fulfills condition {\bf (A2)}. In view of Proposition \ref{p3.2}, the assumption {\bf(A3)} can be achieved.
\begin{lemma}\label{lem5.1}
Let $(u,v,h)$ be the solution of problem $(\ref{a1})$. If $h_\infty<\infty$, then
$h_\infty\leq R^*$.
\end{lemma}
\begin{proof}We easily know from Theorem \ref{the5.1} that
$$\lim_{t\to\infty}\|u(t,\cdot)\|_{C([0,h(t)])}=0$$
and
$$
\lim_{t\to\infty}v(t,r)=V(r)\ \ {\rm uniformly\ on\ any\ compact\ subset\ of}\  [0,\infty),
$$

Assume on the contrary that $R^* <h_\infty< \infty$. By means of Proposition \ref{pr3.2}(ii) we know that $\lambda_1(h_\infty;a_1-c_1V,d_1)<0$. Similar to the proof of Lemma \ref{lem3.1}, there are $0<\varepsilon\ll1$ and $T\gg1$ such that
$$v(t,r)\leq V(r)+\frac{\varepsilon}{c_1(r)},\ \ \forall\ t\geq T, \ 0\leq r\leq h(t),$$
and
$\lambda_1(h(T);a_1-c_1V-\varepsilon,d_1)<0$. Let $w(t,r)$ and $W(r)$ denote the unique positive solution of
\begin{eqnarray*}
\left\{\begin{array}{ll}
w_t-d_1\Delta w=w(a_1(r)-b_1(r)w-c_1(r)V-\varepsilon),&t>T, \ 0\leq r<h(T),\\[1mm]
w_r(t,0)=0=w(t,h(T)),&t\geq T,\\[1mm]
w(T,r)=u(T,r),&0\leq r\leq h(T)
\end{array}
\right.
\end{eqnarray*}
and the corresponding stationary problem
\begin{eqnarray*}
\left\{\begin{array}{ll}
-d_1\Delta W=W(a_1(r)-c_1(r)V-\varepsilon-b_1(r)W),&0\leq r<h(T),\\[1mm]
W_r(0)=0=W(h(T)),
\end{array}
\right.
\end{eqnarray*}
respectively. Then, we deduce that  $\liminf_{t\to\infty}u(t,r)\geq\lim_{t\to\infty}w(t,r)=W(r)>0$ in $(0,h(T))$, which is in contradiction to the fact that $\lim_{t\to\infty}\|u(t,\cdot)\|_{C([0,h(t)])}=0$. This finishes the proof.
\end{proof}
The next lemma can be obtained in a similar manner to Lemma \ref{lem3.3}(ii).
\begin{lemma}\label{lem5.2}
If $h_0<R^*$, then there exists a positive number $\bar\mu$ so that $h_\infty=\infty$ if $\mu>\bar\mu$.
\end{lemma}
The hypothesis (\ref{e0}) ensures that there exists $0<\tilde R<\infty$ such that $\lambda_1(\tilde R;a_1(r), d_1)=0$. By virtue of Proposition \ref{pr3.2}(ii) it follows that $\tilde R<R^*$.
\begin{lemma}\label{lem5.3}
If $h_0<\tilde R$, then there exists a positive number $\underline \mu$ so that $h_\infty<\infty$ when $\mu\leq\underline\mu$.
\end{lemma}
\begin{proof}
Obviously, $(u,h)$ satisfies
\begin{eqnarray*}
\left\{
\begin{array}{ll}
u_t-d_1\Delta u\leq u(a_1(r)-b_1(r)u),&t>0,\ 0\leq r< h(t),\\[1mm]
u_r(t,0)=0,\ u(t,h(t))=0,&t>0,\\[1mm]
h'(t)=-\mu u_r(t,h(t)),&t>0,\\[1mm]
u(0,r)=u_0(r),&0\leq r\leq h_0,
\end{array}
\right.
\end{eqnarray*}
which implies that $(u,h)$ is a lower solution to the problem
\begin{eqnarray*}
\left\{
\begin{array}{ll}
\bar u_t-d_1\Delta \bar u= \bar u(a_1(r)-b_1(r)\bar u),&t>0,\ 0\leq r< \bar h(t),\\[1mm]
\bar u_r(t,0)=0,\ \bar u(t,\bar h(t))=0,&t>0,\\[1mm]
\bar h'(t)=-\mu \bar u_r(t,\bar h(t)),&t>0,\\[1mm]
\bar u(0,r)=u_0(r),\ \bar h_0=h_0,&0\leq r\leq \bar h_0.
\end{array}
\right.
\end{eqnarray*}
Note $h_0<\tilde R$, from the proof of Lemma \ref{lem3.2} (also see Theorem 3.4 in \cite{Wjde14}) it is easy to deduce that there exists $\underline \mu>0$ such that $\bar h_\infty<\infty$ when $\mu\leq\underline\mu$. Making use of the comparison principle for single equation with a free boundary gives $h_\infty<\infty$ when $\mu\leq\underline\mu$.
\end{proof}

To find the sharp criteria governing the alternatives in the spreading-vanishing dichotomy, we require the following comparison principle, which can be argued as in Lemma 2.6 of \cite{DL2}.
\begin{lemma}\label{lem5.4}
Assume that $T\in(0,\infty)$, $\bar h,\underline h\in C^1([0,T])$, $\bar u\in C(\overline{Q_T^*})\cap C^{1,2}(Q_T^*)$ with $Q_T^*=\{(t,r)\in\mathbb{R}^2:t\in(0,T], r\in(0,\bar h(t))\}$,  $\underline u\in C(\overline{Q_T^{**}})\cap C^{1,2}(Q_T^{**})$ with $Q_T^{**}=\{(t,r)\in\mathbb{R}^2:t\in(0,T], r\in(0,\underline h(t))\}$, $\bar v,\underline v\in(L^\infty\cap C)([0,T]\times[0,\infty))\cap C^{1,2}((0,T]\times[0,\infty))$ and
\begin{eqnarray*}
\left\{
\begin{array}{ll}
\bar u_t-d_1\Delta \bar u\geq \bar u(a_1(r)-b_1(r)\bar u-c_1(r)\underline v),&0<t\leq T,\  0\leq r<\bar h(t),\\[1mm]
\underline u_t-d_1\Delta \underline u\leq \underline u(a_1(r)-b_1(r)\underline u-c_1(r)\bar v),&0<t\leq T,\ 0\leq r<\underline h(t),\\[1mm]
\bar v_t-d_2\Delta \bar v\geq \bar v(a_2(r)-b_2(r)\underline u-c_2(r)\bar v),&0<t\leq T,\ 0\leq r<\infty,\\[1mm]
\underline v_t-d_2\Delta \underline v\leq \underline v(a_2(r)-b_2(r)\bar u-c_2(r)\underline v),&0<t\leq T,\ 0\leq r<\infty,\\[1mm]
\bar u_r(t,0)=\underline v_r(t,0)=0,\ \bar u(t,r)=0,&0<t\leq T,\ \bar h(t)\leq r <\infty,\\[1mm]
\underline u_r(t,0)=\bar v_r(t,0)=0,\ \underline u(t,r)=0,&0<t\leq T,\ \underline h(t)\leq r <\infty,\\[1mm]
\bar h'(t)\geq-\mu\bar u_r(t,\bar h(t)),\ \underline h'(t)\leq-\mu\underline u_r(t,\underline h(t)), &0<t\leq T,\\[1mm]
\bar h(0)\geq h_0\geq\underline h(0),\\[1mm]
\bar u(0,r)\geq u_0(r)\geq\underline u(0,r),&0\leq r\leq h_0,\\[1mm]
\bar v(0,r)\geq v_0(r)\geq\underline v(0,r),&0\leq r<\infty.
\end{array}
\right.
\end{eqnarray*}
Let $(u,v,h)$ be the unique solution of $(\ref{a1})$. Then
\begin{eqnarray*}
h(t)\leq\bar h(t)\ {\rm for}\  0<t\leq T,\ u(t,r)\leq\bar u(t,r),\ v(t,r)\geq\underline v(t,r)\ {\rm for }\ 0<t\leq T,\ 0\leq r<\infty,\\[1mm]
h(t)\geq\underline h(t)\ {\rm for }\  0<t\leq T,\ u(t,r)\geq\underline u(t,r),\ v(t,r)\leq\bar v(t,r)\ {\rm for }\  0<t\leq T,\ 0\leq r<\infty.
\end{eqnarray*}
\end{lemma}
We can now use Lemmas \ref{lem5.1}-\ref{lem5.4} to show the following sharp criteria for spreading and vanishing.
\begin{theorem}\label{the5.3}
{\rm (i)} If $h_0\geq R^*$, then $h_\infty=\infty$ for all $\mu>0$.

{\rm (ii)} If $h_0< R^*$, then there exists  $\hat \mu>0$ such that $h_\infty=\infty$ for any $\mu>\hat\mu$, whereas $h_\infty\leq R^*$ for any $\mu\leq\hat\mu$.
\end{theorem}
\begin{proof} We only sketch the proof of assertion(ii), since (i) can be manifested by the same argument as in Theorem \ref{the3.2}(i).

Define $\mathcal{X}=\{\mu>0:\ h_\infty>R^*\}$. In terms of Lemmas \ref{lem5.1} and \ref{lem5.2}, we know that $\hat\mu=\inf\mathcal{X}\in(0,\infty)$. Notice that Lemma \ref{lem5.4} implies the monotonicity of $h_\infty$ with respect to $\mu$. Therefore, it follows from Lemma \ref{lem5.1} that $h_\infty<\infty$ if $\mu<\hat\mu$ and $h_\infty=\infty$ if $\mu>\hat\mu$. Same to the proof for $\mu^*\in\mathcal{X}^*$ in Theorem \ref{the3.2}, it can be derived that $\hat \mu\notin\mathcal{X}$.
\end{proof}
Then, we consider the long time behavior   for the spreading of $u$. Actually, by the argument of Theorem \ref{the4.1} with some obvious modifications we can show
\begin{theorem}\label{the5.4}
Suppose that $\underline{a}_2\underline{b}_1-\bar{a}_1\bar{b}_2>0$, $\underline{a}_1\underline{c}_2-\bar{a}_2\bar{c}_1>0$ and $(u,v,h)$ is the the unique solution of $(\ref{a1})$ with $h_\infty=\infty$. Then
\begin{eqnarray*}
\underline  u(r)\leq\liminf_{t\to\infty}u(t,r), \ \  \limsup_{t\to\infty}u(t,r)\leq\bar u(r)& {\rm\ uniformly\ on\ any\ compact\ subset\ of\ [0,\infty)},\\[1mm]
\underline v(r)\leq\liminf_{t\to\infty}v(t,r),\ \  \limsup_{t\to\infty}v(t,r)\leq\bar v(r)& {\rm\ uniformly\ on\ any\ compact\ subset\ of\ [0,\infty)},
\end{eqnarray*}
where $\underline u(r)$, $\bar u(r)$, $\underline v(r)$ and $\bar v(r)$ are given in the proof of Theorem $\ref{theor4.2}$.
\end{theorem}

Except for these above results corresponding to problem (\ref{1.1}) discussed in Sections \ref{sec2}-\ref{sec4},  we also obtain the asymptotic spreading speed of the free boundary $h(t)$ for problem (\ref{a1}) when $\rho$ is restricted to 0, i.e.,
\begin{equation}\label{e5}
\underline{a}_i=\liminf_{r\to\infty}{a_i(r)},\ \ \ \bar{a}_i=\limsup_{r\to\infty}{a_i(r)}.
\end{equation}

Let us first state the following known consequence, which plays an important role in later discussion. One can find the proof in \cite[Proposition 2.1]{BDK}.
\begin{proposition}
For any given positive constants $a,b,d$ and $k\in [0,2\sqrt{ad})$, the problem
$$-dw''+kw'=aw-bw^2\ \ {\rm in}\ 0< r< \infty,\ \ w(0)=0$$
admits a unique positive solution $w=w_k=w_{a,b,d,k}$, which satisfies $w(r)\to {\frac{a}{b}}$ as $r\to \infty.$ Moreover, $w_k'(r)>0$ for all $r\geq0$, $w_{k_1}'(r)>w_{k_2}'(r)$ for any $r>0$ and $k_1<k_2$, and for each $\mu>0$, there exists a unique $k_0=k_0(\mu, a,b,d)\in [0,2\sqrt{ad})$ such that $\mu w_{k_0}'(0)=k_0$. Furthermore, $$\lim_{\frac{\mu a}{bd}\to\infty}\frac{k_0}{\sqrt{ad}}=2,\ \ \ \lim_{\frac{\mu a}{bd}\to0}\frac{k_0}{\sqrt{ad}}\frac{bd}{\mu a}=\frac{1}{\sqrt 3}. $$
\end{proposition}
Taking advantage of the function $k_0(\mu, a,b,d)$, we can derive the following estimates for the asymptotic spreading speed of $h(t)$.
\begin{theorem}
Suppose that $\underline{a}_1\underline{c}_2-\bar{a}_2\bar{c}_1>0$ and $h_\infty=\infty$. Then
$$k_0(\mu, \underline{a}_1-\bar{a}_2\bar{c}_1/\underline{c}_2,\bar{b}_1,d_1)\leq \liminf_{t\to\infty}\frac{h(t)}{t} \leq\limsup_{t\to\infty}\frac{h(t)}{t} \leq k_0(\mu, \bar{a}_1,\underline{b}_1,d_1). $$
\end{theorem}
\begin{proof}
Because
\begin{eqnarray*}
\left\{
\begin{array}{ll}
u_t-d_1\Delta u=u(a_1(r)-b_1(r)u-c_1(r)v)\leq u(a_1(r)-b_1(r)u),&t>0,\ 0\leq r< h(t),\\[1mm]
u_r(t,0)=0,\ u(t,h(t))=0,&t>0,\\[1mm]
h'(t)=-\mu u_r(t,h(t)),&t>0,\\[1mm]
u(0,r)=u_0(r),&0\leq r\leq h_0.
\end{array}
\right.
\end{eqnarray*}
This indicates that $(u,h)$ is a lower solution to the problem
\begin{eqnarray*}
\left\{
\begin{array}{ll}
\bar u_t-d_1\Delta \bar u= \bar u(a_1(r)-b_1(r)\bar u),&t>0,\ 0\leq r< \bar h(t),\\[1mm]
\bar u_r(t,0)=0,\ \bar u(t,\bar h(t))=0,&t>0,\\[1mm]
\bar h'(t)=-\mu \bar u_r(t,\bar h(t)),&t>0,\\[1mm]
\bar u(0,r)=u_0(r),\ \bar h_0=h_0,&0\leq r\leq \bar h_0.
\end{array}
\right.
\end{eqnarray*}
By means of the comparison principle, it is easy to obtain that $\bar h(t)\geq h(t)$ as $t\to\infty$. A similar argument as in \cite[Theorem 6.1]{ZX} gives rise to
$$\lim_{t\to\infty}\frac{\bar h(t)}{t}=k_0(\mu, \bar{a}_1,\underline{b}_1,d_1).$$
And it then follows that
$$\limsup_{t\to\infty}\frac{h(t)}{t} \leq k_0(\mu, \bar{a}_1,\underline{b}_1,d_1).$$

Next, by constructing a suitable lower solution, we want to show
$$\liminf_{t\to\infty}\frac{h(t)}{t}\geq k_0(\mu, \underline{a}_1-\bar{a}_2\bar{c}_1/\underline{c}_2,\bar{b}_1,d_1).$$
Note  (\ref{e5}) and the assumption ({\bf H}). So, for any $\varepsilon'>0$, there exists $R'=R'(\varepsilon')>0$ such that for any $r\geq R'$,
$$\underline{a}_i-\varepsilon'\leq a_i(r)\leq \bar{a}_i+\varepsilon',\ \ \ \underline{c}_i-\varepsilon'\leq c_i(r)\leq \bar{c}_i+\varepsilon'.$$
Hence it is not difficult to show that
$$\limsup_{t\to\infty}v(t,r)\leq \frac{\bar{a}_2+\varepsilon'}{\underline{c}_2-\varepsilon'}\ \ {\rm uniformly\ for}\ r\in [R',\infty).$$
By virtue of Theorem \ref{the5.4}, we know that $\limsup_{t\to\infty}v(t,r)\leq\bar v(r)$ uniformly on any compact subset of $[0,\infty).$ Thanks to $h_\infty=\infty$, for any given $0<\varepsilon\ll1$, there exists $ R\gg1$, $T=T(\varepsilon)>0$ and positive function $v^*(r)\in C^\gamma([0,\infty))$
such that
\begin{eqnarray*}
 &v(t,r)\leq  v^*(r)+\varepsilon\ \  {\rm for}\  t\geq T{\rm \ and}\ 0\leq r<\infty,&  \\
  &v^*(r)={\bar{a}_2}/{\underline{c}_2}\ \ {\rm for}\   R\leq r<\infty,\  {\rm and}\ h(T)> R.&
\end{eqnarray*}

Consider the following auxiliary problem
\begin{eqnarray*}
\left\{
\begin{array}{ll}
\underline u_t-d_1\Delta \underline u= \underline u[a_1(r)-c_1(r)(v^*(r)+\varepsilon)-b_1(r)\underline u],&t>T,\ 0\leq r< \underline h(t),\\[1mm]
\underline u_r(t,0)=0,\ \underline u(t,\underline h(t))=0,&t>T,\\[1mm]
\underline h'(t)=-\mu \underline u_r(t,\underline h(t)),&t>T,\\[1mm]
\underline  u(T,r)=u(T,r),&0\leq r\leq  h(T).
\end{array}
\right.
\end{eqnarray*}
Apparently, $(u,h)$ is an upper solution of the above problem, and
$$\liminf_{r\to\infty}[a_1(r)-c_1(r)(v^*(r)+\varepsilon)]\geq \underline{a}_1-\bar{c}_1(\bar{a}_2/\underline{c}_2+\varepsilon).$$
Again making use of the argument of \cite{ZX} brings about
$$\lim_{t\to\infty}\frac{\underline h(t)}{t}=k_0(\mu, \underline{a}_1-\bar{c}_1(\bar{a}_2/\underline{c}_2+\varepsilon),\bar{b}_1,d_1)$$
which implies that
$$\liminf_{t\to\infty}\frac{h(t)}{t} \geq k_0(\mu, \underline{a}_1-\bar{c}_1(\bar{a}_2/\underline{c}_2+\varepsilon),\bar{b}_1,d_1).$$
Due to the arbitrariness of $\varepsilon$, the desired result can be derived immediately.
\end{proof}

\section{Discussion}

In this article we have studied the dynamical behavior of the two competing species $u(t,|x|)$ and $v(t,|x|)$ with expanding front $\{|x|=h(t)\}$ determined by $h'(t)=-\mu [u_x(t,h(t))+\beta v_x(t,h(t))]$, i.e., (\ref{1.1}), and also the dynamical behavior of the new competitor $u(t,|x|)$ invading into the native species $v(t,|x|)$ with expanding front $\{|x|=h(t)\}$ determined by $h'(t)=-\mu u_x(t,h(t))$, i.e., (\ref{a1}). We suppose that these species exist in a heterogeneous environment, especially, the variable intrinsic growth rate $a_i(|x|)$ (i=1,2) may be ``very negative'' in the sense that both $\int_{\mathbb{R}^N}a_i(|x|)dx=-\infty$ and  $|\{x : a_i(|x|)>0\}|\ll|\{x: a_i(|x|)<0\}|$ are allowed (see {\bf (A2)}),  where $|A|$ denotes the measure of $A$. That is,  the results in \cite{GW,WZh1,DL2} are extended to the more realistic environment.

From the above discussion we have realized that the number $h^*$ satisfying either $\lambda_1(h^*;a_1,d_1)=0$ or $\lambda_1(h^*;a_2,d_2)=0$ is crucial to problem (\ref{1.1}); To (\ref{a1}) the counterpart is $R^*$. Let $l^*=h^*$, $\nu_*=\mu_*$ and $\nu^*=\mu^*$ for the former problem, and $l^*=R^*$ and $\nu_*=\nu^*=\hat\mu$ for the latter.
We have proved that

(i) If the expanding radius of initial habitat is less than $l^*$ and the moving parameter $\mu$ of the expanding front is less than $\nu_*$, then $h_\infty<l^*$. Moveover,

(ia) for problem (\ref{1.1}), $\lim_{t\to\infty}\|u(t,\cdot)\|_{C([0,h(t)])}=\lim_{t\to\infty}\|v(t,\cdot)\|_{C([0,h(t)])}=0$;

(ib) for problem (\ref{a1}), $\lim_{t\to\infty}\|u(t,\cdot)\|_{C([0,h(t)])}=0$
and
$\lim_{t\to\infty}v(t,r)=V(r)$ uniformly on any compact subset of $[0,\infty)$,
where $V(r)$ is the only positive solution of the problem $-d_2\Delta v=v(a_2(r)- c_2(r)v)\ \ {\rm in \ }\mathbb{R}^N$.

(ii) If the expanding radius of initial habitat is not less than $l^*$, or it is less than $l^*$ but the moving parameter $\mu$ of the expanding front is greater than $\nu^*$, then $h_\infty=\infty$. Moreover,
if $\underline{a}_2\underline{b}_1-\bar{a}_1\bar{b}_2>0$, $\underline{a}_1\underline{c}_2-\bar{a}_2\bar{c}_1>0$, then $u(t,r)$ and $v(t,r)$ satisfy
$$\underline  u(r)\leq\liminf_{t\to\infty}u(t,r),\ \limsup_{t\to\infty}u(t,r)\leq\bar u(r),\ \  \underline v(r)\leq\liminf_{t\to\infty}v(t,r),\  \limsup_{t\to\infty}v(t,r)\leq\bar v(r)$$
uniformly on any compact subset of $[0,\infty)$, where $\underline u(r)$, $\bar u(r)$, $\underline v(r)$ and $\bar v(r)$ are given in the proof of Theorem $\ref{theor4.2}$.

Our conclusions not only provide the sufficient conditions for species spreading success and spreading failure, but also provide the long time behavior of $(u(t,r),v(t,r))$. If the expanding radius of initial habitat is small, and the moving parameter is sufficiently small, it turns out that no population can survive eventually for (\ref{1.1}), and no new competitor $u(t,r)$  can survive for (\ref{a1}). On the other hand, If the expanding radius of initial habitat or the moving parameter is large enough, regardless of initial population size, then the expanding domain inevitably becomes the whole existing space. The phenomenon suggests that the expanding radius of initial habitat and the moving parameter are important to the survival for the species. The better way to reduce the moving parameter may be to control the surrounding environment.

These theoretical results may be helpful in the prediction and prevention of biological invasions.

\end{document}